\documentclass[a4paper,12pt,english,reqno]{amsart}
\usepackage[latin1]{inputenc}
\usepackage{amsmath}
\usepackage{mathabx}
\usepackage{amsfonts}
\usepackage{amssymb}
\usepackage{graphicx}
\usepackage{graphics}
\usepackage{mathabx}
\usepackage{epsfig}
\usepackage{wrapfig}
\usepackage{here}
\usepackage{amscd}
\usepackage{color}
\usepackage{shadow}
\usepackage{a4wide}
\usepackage[all]{xy}
\usepackage{latexsym}
\usepackage{hyperref,amsthm}
\usepackage{amsopn}
\usepackage{epic, eepic}
\usepackage{fancyhdr}
\newtheorem{th1}{Theorem}[section]

\newtheorem{lem}[th1]{ Lemma}

\newtheorem{rem}[th1]{ Remark}

\title[\textbf{Blow-up of semidiscrete solution}]{Blow-up of semidiscrete solution of a nonlinear parabolic equation with gradient term}

\author{Houda Hani} 
\address{Laboratoire de Math\`ematiques: \\ Mod\'elisation D\'eterministe et Al\'eatoire\\ 
 universit\'e de sousse,\\
 houda.hani84@gmail.com}

\author{Moez Khenissi}
\address{Laboratoire de Math\`ematiques: \\ Mod\'elisation D\'eterministe et Al\'eatoire\\ 
 universit\'e de sousse,
	moez.khenissi@essths.u-sousse.tn}

\begin{document}

	\begin{abstract}
		This paper is concerned with approximation of blow-up phenomena in nonlinear parabolic problems. We consider the equation $u_{t}=u_{xx}+\left| u\right| ^{p}-b(x)\left|u_{x}\right|^{q}$ in a bounded domain, we study the behavior of the semidiscrete problem. Under some assumptions we show existence and unicity of the semidiscrete solution, we show that it blows up in a finite time and we prove the convergence of the semidiscrete problem. Finally, we give an approximation of the blow up rate and the blow up time of the semidiscrete solution.  
	\end{abstract}
	
	\maketitle
	\tableofcontents
	\section{Introduction}
    \noindent	There is a large number of nonlinear partial differential equations of parabolic type whose solution for a given initial data cannot be extended globally in time and becomes unbounded in finite time. Such a phenomenon is called blow-up, this can occur in nonlinear equations if the heat source is strong enough.\\
    
	In this paper, we consider the initial boundary value problem for the following nonlinear parabolic partial differential equation
	\begin{equation}
	\left\{ 
	\begin{array}{llll}
	u_{t}=\Delta u+\left| u\right| ^{p}-b\left|\nabla u\right|^{q}  \text {  in   } \Omega\times(0,T^{*}),\\
	u(x,0)=u_{0}(x)  \text {  for   } x\in \Omega,\\
	u(x,t)=0 \text { for  } x\in \partial\Omega,\ \ t\in (0,T^{*}), \\
	\end{array}
	\right.
	\label{exacte}
	\end{equation}
	in a bounded domain $\Omega\subset \mathbb{R}^{d}$, for $p,\ q>1$ and $b$ is a continuous positive and bounded function.\\
	Here $\left( 0,T^{*}\right) $ is the maximal time interval on which $\left\|u(.,t)\right\|_{\infty}:=\max\limits_{x\in\Omega}|u(x,t)|<\infty.$ The time $T^{*}$ may be finite or infinite. When $T^{*}$ is infinite, we say that the solution $u$ exists globally. When $T^{*}$ is finite, then we have $$\lim\limits_{t\longrightarrow T^{*}}\left\| u(.,t)\right\| _{\infty}=+\infty.$$ In this case, we say that the solution $u$ blows up in a finite time and $T^{*}$ is called the blow up time.\\
	The above problem was introduced by Chipot and Weissler in 1989 (\cite{chipotweissler}) in the case $b=1$. They have proved local existence, uniqueness and regularity for the problem in a bounded domain $\Omega \subset \mathbb{R}^{d}$. They showed that for $s$ sufficiently large  $$u\in C^{1}\left((0,T^{*});W_{0}^{1,s}(\Omega)\right)\bigcap C\left((0,T^{*});W^{2,\frac{s}{q}}(\Omega)\right),$$  $\left\|u(t)\right\|_{\infty}$ and $\left\|\nabla u(t)\right\|_{\infty}$ are bounded on any interval $[0,T]$ with $T<T^{*}.$\\
 When $b$ is positive constant, problem \eqref{exacte} is related to a popular model arising in the study of the dynamic of population (see \cite{souplet}).  There has been many works in the past concerning numerical computation of solutions of nonlinear parabolic equation but without the gradient term (see \cite{hamada}, \cite{cho}, \cite{nakagawa} and \cite{Chen}). Note that the gradient term has a damping effect working against blow up.\\
	The theoretical study of blowing up solutions of \eqref{exacte} when $b$ is a constant, has been the subject of investigations of many authors (see \cite{cheblikquittner}, \cite{souplet}, \cite{ souplettayachi}, \cite{soupletweissler}, \cite{snoussi} and the references therein). In particular, in \cite{souplettayachi}, the author has proved that under some assumptions on $p$, $q$, the initial data and $b$, the solution of \eqref{exacte} blows up in a finite time. They proved the next theorem 
	\begin{th1}
	Assume $p>1$ and $1<q\leq \dfrac{2p}{p+1}$, $u_{0}$ sufficiently regular satisfies 
	\begin{equation*}
	E(u_{0})=\dfrac{1}{2}\left\|\nabla u_{0}\right\|_{2}^{2}-\dfrac{1}{p+1}\left\|u_{0}\right\|^{p+1}_{p+1}<0.
	\end{equation*}
	Let $u$ be the solution of \eqref{exacte} such that $u_{t}\geq 0.$ Moreover, suppose that $\dfrac{-E(u_{0})}{\left\|u_{0}\right\|^{2}_{2}}$ is large enough if $q<\dfrac{2p}{p+1}$, or that $b$ is sufficiently small if $q=\dfrac{2p}{p+1}.$\\
	Then the solution of \eqref{exacte} blows up in a finite time.
	\end{th1}

\noindent In \cite{khe} and \cite{hani} we have studied the problem for $b=1$, we have constructed a finite difference scheme which approximate the exact problem \eqref{exacte} (for $b=1$). We have showed that under some assumptions on $p$, $q$ and the initial data, the numerical solution blows up in a finite time and we have estimated the numerical blow-up time. We have also proved that although the exact solution blows up in one point, the numerical solution blows up in more than one point under some assumptions on $p$ and $q$.\\

\noindent In \cite{alfonsi}, authors proved that blow up in finite time occurs for $1<q\leq \dfrac{2p}{p+1}$. In particular, for $q=\dfrac{2p}{p+1}$ and $b$ is a small real such that $0\leq b< \dfrac{p-1}{2}\bigg(\dfrac{2}{p+1}\bigg)^{1/p+1}$, solution blows up in finite time for a positive initial data $u_0$ sufficiently regular satisfying:
$E(u_0)<0\ \ and \ \  \Delta u_0+ u_0 ^{p}-b\left|\nabla u_0\right|^{q}\geq 0 $
(this last assumption ensuring the positiveness of $u$ and $u_t$ for all $t$). However, the existence of such initial data is guaranteed only for $d=1$ and $p<5$.\\
In \cite{kawohl}, Kawohl and Peletier, showed that the gradient damping term prevents blow up if $1<p\leq q=2$.\\

In \cite{angenent} and \cite{soup} authors have considered the next problems
$$u_t-\Delta u=a(x)u^p+\left|\nabla u\right|^{q}$$
$$u_t- u_{xx}=f(u)\left| u_x\right|^{q-1}u_x$$
They showed that under some assumptions on the initial data, $\left|\left|a\right|\right|_\infty$, $p$, $q$ and $f$, we have blow up of the gradient of the solution $u$.\\
To our knowledge, there are no theoretical nor numerical results concerning the case where $b$	is a function independent of the solution $u$. All the theoretical study of \eqref{exacte} concerns only the case where $b$ is a positive constant.\\

\noindent In this paper, we are interested in the numerical study of the above problem using the semidiscrete form defined in \eqref{Eh}. The semi-discretization in space of \eqref{exacte} leads to an initial value problem for a system of nonlinear ordinary differential equations. We give some assumptions under which the solution of \eqref{Eh} blows up in a finite time and we estimate the numerical blow-up time and the numerical blow up rate.\\
Based on these numerical results, the theoretical study of \eqref{exacte} will be studied in a future paper.\\
	
\noindent	Our paper is written in the following manner. In the next section, we prove some results about the semidiscrete solution. In the third section, under some hypotheses, we show that the solution of the semidiscrete problem blows up in a finite time. In the fourth section, we give a result about the convergence of the semidiscrete solution to the theoretical one when the mesh size goes to zero. In section 5, we give an approximation of the blow up rate. In section 6, we give an estimate of the blow up time of the semidiscrete solution . Finally, in the last section, we present some numerical experiments.
	\section{The semidiscrete problem}
\noindent	We consider the semilinear parabolic equation
	\begin{equation}
	\left\{ 
	\begin{array}{llll}
	u_{t}=u_{xx}+\left| u\right| ^{p}-b\left|u_{x}\right|^{q}  \text {  in   } (-1,1)\times(0,T^{*}),\\
	u(x,0)=u_{0}(x)  \text {  for   } x\in (-1,1),\\
	u(-1,t)=u(1,t)=0 \text { for  } t\in (0,T^{*}),  \\
	\end{array}
	\right.
	\label{E}
	\end{equation}
	where $p>1$, $1< q\leq \dfrac{2p}{p+1}$ and $b$ is a continuous, positive and bounded function.\\
    The initial data $u_{0}$ is a continuous, nonconstant and nonnegative function in $[-1,1]$. We suppose also that $\left\|u_{0}\right\|_{\infty}$ is large enough.
	\subsection{Definition of the semidiscrete problem}
	Let $N$ be a positive integer representing the number of subdivisions of the interval $[-1,1]$ and $h$ the spacial mesh size defined below such that $N=E\left(\dfrac{2}{h}\right)+1$, where $E(X)$ is the integer part of $X$.\\
	We define the grid $x_{j}=-1+jh$ for $0\leq j \leq N+1$, and we approximate the solution $u$ of \eqref{E} by $U_{h}(t)=(u_{0}(t),...,u_{N+1}(t))'$.\\
	Spacial discretization of \eqref{E} yields 
	\begin{equation}
	\left\{ 
	\begin{array}{lllll}
	\dfrac{du_{j}(t)}{dt}-\delta^{2}_{x}u_{j}(t)+b_j\left|\delta_{x}u_{j}(t)\right|^{q}=\left| u_{j}(t)\right| ^{p}, \ \ \  t\in(0,T^*_{h})  \text { and } 1\leq j \leq N,\\
u_{j}(0)=u_{j}^{0}\geq0 \text{ for }1\leq j \leq N,\\
u_{0}(t)=u_{N+1}(t)=0,\ \ \ t\in (0,T^*_{h}).
	\end{array}
	\right.\\
	\label{Eh}
	\end{equation} 
	Here we define:
	\begin{itemize}
		\item $b_j$  the approximation of $b(x_j)$,		
\item 	 

$	\delta^{2}_{x}u_{j}=\dfrac{u_{j+1}-2u_{j}+u_{j-1}}{h^{2}}$ an approximation of $u_{xx}$ 
	 \item $\delta^{+}_{x}u_{j}=\frac{u_{j+1}-u_{j}}{h},\ \ \,\delta^{-}_{x}u_{j}=\frac{u_{j}-u_{j-1}}{h}\ \ and \ \ 	\delta_{x}u_{j}=\dfrac{u_{j+1}-u_{j-1}}{2h}=\dfrac{\delta^{+}_{x}u_{j}+\delta^{-}_{x}u_{j}}{2} $ approximations of $u_{x}$.\\
			\end{itemize}
Here $(0,T^*_{h})$  is the maximal time interval on which $\left\|U_{h}(t)\right\|_{\infty}<\infty$ with
					$$\left\| U_{h}(t)\right\| _{\infty}=\max\limits_{j=0,...,N+1}\left| u_{j}(t)\right|. $$ 
\begin{enumerate}
 \item If $T^*_{h}=+\infty$ then $U_{h}$ is a global solution.\\
\item If $T^*_{h}<+\infty$ we say that the solution $U_{h}$ achieves blow up in a finite time and we have
\begin{equation*}
\left\| U_{h}(t)\right\| _{\infty}<\infty \text{ for } t\in [0,T^*_{h}) \text { but } \lim_{t\rightarrow T^*_{h}}	\left\| U_{h}(t)\right\| _{\infty}=\infty.
\end{equation*}
 In this case, the time $T^*_{h}$ is called the numerical blow-up time of the solution $U_{h}(t)$.
\end{enumerate} 
\vspace{0.5cm}
Let $T_h<T^*_{h}$, $h_{0}>0$ sufficiently small and $M_{T_h}:=\max\limits_{\underset{0\leq j \leq N+1}{0\leq t\leq T_h}}\left|  u_{j}(t)\right| $ which is bounded before blow-up, then we define the spacial mesh size by
$$h=\min\left(h_{0},\left(\dfrac{2}{b_{\infty}q}M_{T_h}^{-q+1} \right)^{\frac{1}{2-q}}\right)$$
where $b_{\infty}$ denotes $\left\| b\right\| _{\infty}$.\\
We define the $l^{\alpha}$ norm of the numerical solution by
\begin{equation*}
\left\| U_{h}(t)\right\| _{\alpha}=\left( \sum\limits_{j=1}^{N}h\left| u_{j}(t)\right| ^{\alpha}\right) ^{\frac{1}{\alpha}} \text{ for all } \alpha\geq 1 \ \text{and} \ t\in [0,T^*_{h}).
\end{equation*}
Let $T<T^{*}$, we denote by
\begin{equation*}
\vvvert{u}\vvvert =\max\limits_{\underset{x\in[-1,1]}{t\in[0,T]}}\left| u(x,t)\right|.
\end{equation*}
the $L^{\infty}$ norm of the exact solution of \eqref{E} in $[-1,1]\times [0,T]$.\\

\noindent	In this section, we give some properties of the semidiscrete solution.
\subsection{Properties of the semidiscrete solution}
The next lemma shows the positivity of the semidiscrete solution.
\begin{lem}
Let  $U_{h}\in C^{1}\left((0,T^*_{h}),\mathbb{R}^{N+2}\right) $ be the solution of  \eqref{Eh}  with initial data $U_{h}^{0}$. 
If $U_{h}^{0}\geq 0$ then $U_{h}(t)\geq 0$ for all $t\in (0,T^*_{h}).$
\end{lem}
\begin{proof} 
		The proof is inspired from \cite{nabongo}. Let $T_h<T^*_{h}$ and $m=\min\limits_{\underset{0\leq j \leq N+1}{0\leq t\leq T_h}} u_{j}(t).$\\
		Since for $0\leq j \leq N+1$, $u_{j}$ is a continuous function, there exists $t_{0} \in (0,T_{h})$ such that $m=u_{j_{0}}(t_{0})$ for a certain $j_{0}\in \left\{0,...,N+1\right\}.$\\
		Assume that $m<0.$\\
		If $j_{0}=0$ or $j_{0}=N+1$, we have a contradiction because $u_{0}(t)=u_{N+1}(t)=0$ for all $t\in[0,T_{h}).$\\
		If $j_{0}\in \left\{1,...,N\right\}$, it is not hard to see that\\
		\begin{equation}
		\dfrac{du_{j_{0}}}{dt}(t_{0})=\lim_{k\rightarrow 0}\dfrac{u_{j_{0}}(t_{0})-u_{j_{0}}(t_{0}-k)}{k}<0
		\label{1etoile}
		\end{equation}
		and 
		\begin{equation}
		\delta_{x}^{2}u_{j_{0}}(t_{0})=\dfrac{u_{j_{0}-1}(t_{0})-2u_{j_{0}}(t_{0})+u_{j_{0}+1}(t_{0})}{h^{2}}\geq 0.
		\label{2etoile}
		\end{equation}
		Define the vector $Z_{h}(t)=e^{\lambda t}U_{h}(t)$ where $\lambda <0$ such that $|\lambda|$ is large enough and 
		\begin{equation*}
		b_{j_0}\left|\dfrac{u_{j_{0}+1}(t_{0})-u_{j_{0}-1}(t_{0})}{2h}\right|^{q}-\lambda u_{j_{0}}(t_{0})<0.
		\end{equation*}
		Using \eqref{1etoile} and \eqref{2etoile} we obtain
		\begin{equation*}
		\dfrac{dz_{j_{0}}}{dt}(t_{0})=\lim_{k\rightarrow 0}\dfrac{z_{j_{0}}(t_{0})-z_{j_{0}}(t_{0}-k)}{k} \leq \lim_{k\rightarrow 0} e^{\lambda t_{0}} \dfrac{u_{j_{0}}(t_{0})-u_{j_{0}}(t_{0}-k)}{k}<0
		\end{equation*}
		and 
		\begin{equation*}
		\delta_{x}^{2}z_{j_{0}}(t_{0})=\dfrac{z_{j_{0}-1}(t_{0})-2z_{j_{0}}(t_{0})+z_{j_{0}+1}(t_{0})}{h^{2}}=e^{\lambda t_{0}} \delta_{x}^{2}u_{j_{0}}(t_{0})\geq 0,
		\end{equation*}
		which implies that
		\begin{equation}
		\dfrac{dz_{j_{0}}}{dt}(t_{0})-\delta_{x}^{2}z_{j_{0}}(t_{0})+e^{\lambda t_{0}}\left(b_{j_0}\left|\dfrac{u_{j_{0}+1}(t_{0})-u_{j_{0}-1}(t_{0})}{2h}\right|^{q}-\lambda u_{j_{0}}(t_{0})\right)<0.
		\label{eq4}        
		\end{equation}
		On the other hand we have
		\begin{equation}
		\dfrac{du_{j_{0}}}{dt}(t_{0})-\delta_{x}^{2}u_{j_{0}}(t_{0})+b_{j_0}\left|\dfrac{u_{j_{0}+1}(t_{0})-u_{j_{0}-1}(t_{0})}{2h}\right|^{q}=\left|u_{j_{0}}(t_{0})\right|^{p}\geq 0,
		\label{a}
		\end{equation}
		but
		\begin{equation*}
		\dfrac{du_{j_{0}}}{dt}(t_{0})=-\lambda e^{-\lambda t_{0}}z_{j_{0}}(t_{0})+e^{-\lambda t_{0}}\dfrac{dz_{j_{0}}}{dt}(t_{0})=-\lambda u_{j_{0}}(t_{0})+e^{-\lambda t_{0}}\dfrac{dz_{j_{0}}}{dt}(t_{0}),
		\end{equation*}
		and
		\begin{equation*}
		\delta_{x}^{2}u_{j_{0}}(t_{0})=e^{-\lambda t_{0}}\delta{x}^{2}z_{j_{0}}(t_{0}),
		\end{equation*}
		then \eqref{a} implies
		\begin{equation*}
		-\lambda e^{-\lambda t_{0}}z_{j_{0}}(t_{0})+e^{-\lambda t_{0}}\dfrac{dz_{j_{0}}}{dt}(t_{0})-e^{-\lambda t_{0}}\delta{x}^{2}z_{j_{0}}(t_{0})+b_{j_0}\left|\dfrac{u_{j_{0}+1}(t_{0})-u_{j_{0}-1}(t_{0})}{2h}\right|^{q}\geq 0,
		\end{equation*}
		and so
		\begin{equation*}
		\dfrac{dz_{j_{0}}}{dt}(t_{0})-\delta^{2}z_{j_{0}}(t_{0})+e^{\lambda t_{0}}\left(b_{j_0}\left|\dfrac{u_{j_{0}+1}(t_{0})-u_{j_{0}-1}(t_{0})}{2h}\right|^{q}-\lambda u_{j_{0}}(t_{0})\right)\geq 0,
		\end{equation*}
		which is a contradiction because of \eqref{eq4}.
	\end{proof}
      \subsection{Existence and unicity of the semidiscrete solution}
      In this section, we prove existence and unicity of the semidiscrete solution of \eqref{Eh}.
        \begin{th1}
        	For all $p>1$ and $1<q\leq \dfrac{2p}{p+1}$ problem \eqref{Eh} has a unique maximal solution $U_{h}\in C^{1}((0,T^*_{h}),\mathbb{R}^{N+2})$. 
        \end{th1}
         To prove the theorem we need the next lemma
         \begin{lem}
         	Let $m\geq 1$ and $\alpha,\ \beta \in\mathbb{R}$. Then we have
         	\begin{equation}
         	\left| \left| \alpha \right| ^{m}-\left| \beta \right| ^{m}\right|\leq m \left| \alpha-\beta \right| \left( \left| \alpha \right| ^{m-1}+\left| \beta \right| ^{m-1} \right) 
         	\label{inegalite} 
         	\end{equation} 
         \end{lem}           
        \begin{proof}
        	 According to the Cauchy-Lipschitz theorem, we know that existence and unicity of solution of \eqref{Eh} hold if the nonlinear term is a locally Lipschitz function.\\
        	 \noindent Let $f(X)=\delta_{x}^{2}X+\left|X\right|^{p}-b\left|\delta_{x}X \right|^{q}$, we shall prove that $f$ is a locally Lipschitz function.  	 
        \noindent Let $X^{*}\in  \mathbb{R}^{N+2}$ such that $X^{*}_{0}=X^{*}_{N+1}=0$, we denote $B_{(X^{*},r)}$ the ball with center $X^{*}$ and radius $r$. Let $X,\  Y\in B_{(X^{*},r)}$,
        then
        	\begin{eqnarray}
        	\left\|f(X)-f(Y) \right\|_{2}&=&\left\|\delta_{x}^{2}X-\delta_{x}^{2}Y+\left|X\right|^{p}-\left|Y\right|^{p}-b\left( \left|\delta_{x}X \right|^{q}- \left|\delta_{x}Y \right|^{q}\right)  \right\|_{2} \nonumber \\
        	&\leq & \left\|\delta^{2}_{x}X-\delta^{2}_{x}Y\right\|_{2}+ \left\|\left|X\right|^{p}-\left|Y\right|^{p}\right\|_{2}+b_{\infty}\left\|  \left|\delta_{x}X \right|^{q}- \left|\delta_{x}Y \right|^{q} \right\|_{2}
        	\label{1star}
        	\end{eqnarray}
        	A straightforward calculation yields
        	\begin{equation}
        \left\|\delta^{2}_{x}X-\delta^{2}_{x}Y\right\|_{2}\leq \dfrac{4}{h}\left\| X-Y\right\|_{2}.
        	\label{delta}
        	\end{equation}
        	Using \eqref{inegalite} we get 
        	\begin{equation}
        	\left\|\left|X\right|^{p}-\left|Y\right|^{p}\right\|_{2}\leq p \left(  { \left\|X\right\|_{\infty}^{p-1}+\left\|Y\right\|_{\infty} ^{p-1}}\right) \left\| X-Y\right\|_{2}  
        	\label{2star}
        	\end{equation}
        	and
        	\begin{equation*}
        	\left\|\left| \delta_{x}X\right| ^{q}-\left| \delta_{x}Y\right| ^{q}\right\|_{2}\leq
        	 \dfrac{q}{h^{q-1}}\left( \left\|X\right\|_{\infty}^{q-1}+\left\|Y\right\|_{\infty}^{q-1}\right)  \left\| \left| \delta_{x}X\right|-\left| \delta_{x}Y\right|\right\|_{2} . 
        	\end{equation*}
        	But
        	\begin{eqnarray*}
        	&&\left\| \left| \delta_{x}X\right|-\left| \delta_{x}Y\right|\right\|_{2}^{2}\\
        	&=& h \sum\limits_{j=1}^{N}{\left(\left|\dfrac{X_{j+1}-X_{j-1}}{2h} \right|- \left|\dfrac{Y_{j+1}-Y_{j-1}}{2h} \right| \right)^{2}}\\
        	&\leq& \dfrac{1}{4h}\sum\limits_{j=1}^{N}{\left((X_{j+1}-X_{j-1})-(Y_{j+1}-Y_{j-1})\right)^{2}}\\
        	&=& \dfrac{1}{4h}\sum\limits_{j=1}^{N}{\left((X_{j+1}-Y_{j+1})-(X_{j-1}-Y_{j-1})\right)^{2}}\\
        	&=&\dfrac{1}{4h}\left(\sum\limits_{j=1}^{N}{(X_{j+1}-Y_{j+1})^{2}}+\sum\limits_{j=1}^{N}{(X_{j-1}-Y_{j-1})^{2}}-2\sum\limits_{j=1}^{N}{(X_{j+1}-Y_{j+1})(X_{j-1}-Y_{j-1})}\right)
           \end{eqnarray*}
           Using that $-2\alpha\beta \leq \alpha^{2}+\beta^{2}$ we get
           	\begin{eqnarray*}
           	\left\| \left| \delta_{x}X\right|-\left| \delta_{x}Y\right|\right\|_{2}^{2}
        	&\leq& \dfrac{1}{2h}\sum\limits_{j=1}^{N}{\left(X_{j+1}-Y_{j+1}\right)^{2}}+\dfrac{1}{2h} \sum\limits_{j=1}^{N}{\left(X_{j-1}-Y_{j-1}\right)^{2}}\\
        	&\leq& \dfrac{1}{h^{2}}\sum\limits_{j=1}^{N}{h\left(X_{j}-Y_{j}\right)^{2}}\\
        	&=& \dfrac{1}{h^{2}} \left\|X-Y\right\|_{2}^{2}. 
        	\end{eqnarray*}
        	which implies that 
        	\begin{equation}
        	\left\|\left| \delta_{x}X\right| ^{q}-\left| \delta_{x}Y\right| ^{q}\right\|_{2}\leq  \dfrac{q}{h^{q}}\left( \left\|X\right\|_{\infty}^{q-1}+\left\|Y\right\|_{\infty}^{q-1}\right)  \left\|X-Y\right\|_{2}.
        	\label{3star}
        	\end{equation}
        	Finally \eqref{1star}, \eqref{delta}, \eqref{2star} and \eqref{3star} implies 
        	\begin{equation*}
        	\left\|f(X)-f(Y) \right\|_{2}\leq
        	 \left(
        	 \dfrac{4}{h}+p \left(  { \left\|X\right\|_{\infty}^{p-1}+\left\|Y\right\|_{\infty} ^{p-1}}\right)+ \dfrac{qb_{\infty}}{h^{q}}\left( \left\|X\right\|_{\infty}^{q-1}+\left\|Y\right\|_{\infty}^{q-1}\right)
        	   \right)
        	     \left\|X-Y\right\|_{2}.
        	\end{equation*}
            Using that $X,\ Y\in B_{(X^{*},r)}$ and $\left\|X\right\|_{\infty}\leq\dfrac{1}{\sqrt{h}}\left\|X\right\|_{2}$ we get
               \begin{eqnarray*}
            K_{h}&:=&\dfrac{4}{h}+p \left(  { \left\|X\right\|_{\infty}^{p-1}+\left\|Y\right\|_{\infty} ^{p-1}}\right)+ \dfrac{qb_{\infty}}{h^{q}}\left( \left\|X\right\|_{\infty}^{q-1}+\left\|Y\right\|_{\infty}^{q-1}\right)\\
            &\leq& \dfrac{4}{h}+\left(\dfrac{p}{h^{\frac{p-1}{2}}}+\dfrac{qb_{\infty}}{h^{\frac{3q-1}{2}}}\right) \left(  { \left\|X\right\|_{2}^{p-1}+\left\|Y\right\|_{2} ^{p-1}}\right)\\
            &\leq& \dfrac{4}{h}+\left(\dfrac{p}{h^{\frac{p-1}{2}}}+\dfrac{qb_{\infty}}{h^{\frac{3q-1}{2}}}\right) \left( 2  \left\|X^{*}\right\|_{2}+r \right)^{p-1}\\
            &:=& L_{h}
            \end{eqnarray*}
            hence 
        		\begin{equation*}
        		\left\|f(X)-f(Y) \right\|_{2}\leq L_{h} \left\|X-Y\right\|_{2}.
        		\end{equation*}
        	which implies that $f$ is a locally lipschitz function. Finally using the Cauchy-Lipschitz theorem we get existence and unicity of the maximal solution of \eqref{Eh}.
        \end{proof}
	\section{Blow up of the semidiscrete solution}
	\noindent Next, we suppose that $p>1$ and $1<q \leq \dfrac{2p}{p+1}$. To prove the blow-up of the semidiscrete solution we need the next lemmas. The first lemma reveals that the solution $U_{h}$ is nondecreasing in time.
		\begin{lem}
		 Let $U_{h}$ be the nonnegative solution of \eqref{Eh} and we suppose that the initial data satisfies $\dfrac{dU_{h}}{dt}(0):=\delta^{2}_{x}U_{h}(0)+U^{p}_{h}(0)-\left|\delta_{x}U_{h}(0)\right|^{q}\geq 0$. Then we have $\dfrac{dU_{h}}{dt}(t)\geq 0$ for all $t\in (0,T^*_{h})$.
		 \label{lemcroissance} 
		\end{lem}
			\begin{proof}
Let $T_h<T^*_{h}$. In the first step of the proof we shall prove that for an initial data $u_{0}$ satisfiying $\delta^{2}_{x}U_{h}(0)+U^{p}_{h}(0)-\left|\delta_{x}U_{h}(0)\right|^{q}\geq 0$, there exists $0<t_{h}<T_h$ such that $\dfrac{dU_{h}}{dt}(t)\geq 0$ for all $t\in [0,t_{h}].$\\
Let $V_{h}(t)=\dfrac{dU_{h}}{dt}(t)$, then for all $j=1,...,N$ we have
			\begin{eqnarray}
	\dfrac{dv_{j}}{dt}(t)&=&\dfrac{v_{j+1}(t)-2v_{j}(t)+v_{j-1}(t)}{h^{2}}+pv_{j}(t)u_{j}^{p-1}(t)\\ \nonumber 
	&&-b_jq\left|\dfrac{u_{j+1}(t)-u_{j-1}(t)}{2h} \right|^{q-2}\dfrac{u_{j+1}(t)-u_{j-1}(t)}{2h}  \dfrac{v_{j+1}(t)-v_{j-1}(t)}{2h} 
			\label{houd}
			\end{eqnarray} 
			Note that $\left|\delta_{x}U_{h}\right|^{q-2}\delta_{x}U_{h}=0$ in case $\left|\delta_{x}U_{h}\right|=0$, this presents no problem since $q>1$.\\
			Now we multiply \eqref{houd} by $v^{-}_{j}=\max(0,-v_{j})$ and we use that 
			\begin{equation*}
		v_{j}^{+}=\max(0,v_{j}),\ v_{j}=v_{j}^{+}-v_{j}^{-} \text { and } v_{j}^{+}v_{j}^{-}=0
			\end{equation*} 
we get
\begin{eqnarray*}
	&&\dfrac{dv_{j}}{dt}(t)v_{j}^{-}(t)=\dfrac{v_{j+1}(t)-v_{j}(t)}{h^{2}}v_{j}^{-}(t)-\dfrac{v_{j}(t)-v_{j-1}(t)}{h^{2}}v_{j}^{-}(t)+pv_{j}(t)v_{j}^{-}(t)u_{j}^{p-1}(t)\\
	&&\ \ \ \ \ \ \ \ \ \ \ \ \ \ -b_jq\left|\dfrac{u_{j+1}(t)-u_{j-1}(t)}{2h} \right|^{q-2} \dfrac{u_{j+1}(t)-u_{j-1}(t)}{2h} \dfrac{v_{j+1}(t)-v_{j-1}(t)}{2h} v_{j}^{-}(t)\\
	&&\Rightarrow
 \dfrac{-1}{2}\dfrac{d}{dt}\left(v_{j}^{-}(t)\right)^{2}=\dfrac{v_{j+1}^{+}(t)}{h^{2}}v_{j}^{-}(t)-\dfrac{v_{j+1}^{-}(t)-v_{j}^{-}(t)}{h^{2}}v_{j}^{-}(t)+\dfrac{v_{j-1}^{+}(t)}{h^{2}}v_{j}^{-}(t)+\dfrac{v_{j}^{-}(t)-v_{j-1}^{-}(t)}{h^{2}}v_{j}^{-}(t)\\
	&&\ \ \ \ \ \ \ \ \ \ \ \ \ \  -p\left(v_{j}^{-}(t)\right)^{2}u_{j}^{p-1}(t)-b_jq\left|\dfrac{u_{j+1}(t)-u_{j-1}(t)}{2h} \right|^{q-2} \dfrac{u_{j+1}(t)-u_{j-1}(t)}{2h} \dfrac{v_{j+1}(t)-v_{j-1}(t)}{2h} v_{j}^{-}(t)\\
\end{eqnarray*}
We sum for $j=1,...,N$ and we use that 
\begin{equation*}
\sum\limits_{j=1}^{N}{v_{j-1}^{+}(t)v_{j}^{-}(t)}=\sum\limits_{j=1}^{N}{v_{j}^{+}(t)v_{j+1}^{-}(t)}
\end{equation*}
we get 
\begin{eqnarray*}
	&& \dfrac{-1}{2}\dfrac{d}{dt}\sum\limits_{j=1}^{N}\left(v_{j}^{-}\right)^{2}(t)=\sum\limits_{j=1}^{N}\dfrac{v_{j+1}^{+}(t)v_{j}^{-}(t)+v_{j}^{+}(t)v_{j+1}^{-}(t)}{h^{2}}-\sum\limits_{j=1}^{N}\dfrac{v_{j+1}^{-}(t)-v_{j}^{-}(t)}{h^{2}}v_{j}^{-}(t)\\
	&&\ \ \ \ \ \ \ \ \ \ \ \ \ \ \ \ \ \ \ \ \ \ \ \ \ \ +\sum\limits_{j=1}^{N}\dfrac{v_{j}^{-}(t)-v_{j-1}^{-}(t)}{h^{2}}v_{j}^{-}(t) -p\sum\limits_{j=1}^{N}\left(v_{j}^{-}\right)^{2}(t)u_{j}^{p-1}(t)\\
	&&\ \ \ \ \ \ \ \ \ \ \ \ \ \ \ \ \ \ \ \ \ \ \ \ \ \ \ -q\sum\limits_{j=1}^{N}b_j\left|\dfrac{u_{j+1}(t)-u_{j-1}(t)}{2h} \right|^{q-2} \dfrac{u_{j+1}(t)-u_{j-1}(t)}{2h} \dfrac{v_{j+1}(t)-v_{j-1}(t)}{2h} v_{j}^{-}(t)\\
	&&\Rightarrow \dfrac{1}{2}\dfrac{d}{dt}\sum\limits_{j=1}^{N}\left(v_{j}^{-}\right)^{2}(t)\leq -\sum\limits_{j=1}^{N}\dfrac{v_{j+1}^{+}(t)v_{j}^{-}(t)+v_{j}^{+}(t)v_{j+1}^{-}(t)}{h^{2}}-\sum\limits_{j=1}^{N}\left(\dfrac{v_{j+1}^{-}(t)-v_{j}^{-}(t)}{h}\right)^{2}-\left(\dfrac{v_{1}^{-}(t)}{h}\right)^{2}\\ 
	&&\ \ \ \ \ \ \ \ \ \ \ \ \ \ \ \ \ \ \ \ \ \ \ \ \  +p\sum\limits_{j=1}^{N}\left(v_{j}^{-}\right)^{2}(t)u_{j}^{p-1}(t) +q\sum\limits_{j=1}^{N}b_j\left|\dfrac{u_{j+1}(t)-u_{j-1}(t)}{2h} \right|^{q-1} \left| \dfrac{v_{j+1}(t)-v_{j-1}(t)}{2h} \right|v_{j}^{-}(t)\\
\end{eqnarray*}
Now, we use that $M_{h}:=\max\limits_{\underset{0\leq j \leq N+1}{0\leq t\leq t_{h}}}\left|  u_{j}(t)\right| $ is bounded before blow up, we can write that 
\begin{equation*}
u_{j}^{p-1}(t)\leq M_{h}^{p-1} \text{ and } \left|\dfrac{u_{j+1}(t)-u_{j-1}(t)}{2h} \right|^{q-1}\leq \left(\dfrac{M_{h}}{h}\right)^{q-1} 
\end{equation*}
and so
\begin{eqnarray*}
&&\dfrac{1}{2}\dfrac{d}{dt}\sum\limits_{j=1}^{N}\left(v_{j}^{-}(t)\right)^{2}\leq -\sum\limits_{j=1}^{N}\dfrac{v_{j+1}^{+}(t)v_{j}^{-}(t)+v_{j}^{+}(t)v_{j+1}^{-}(t)}{h^{2}}-\sum\limits_{j=1}^{N}\left(\dfrac{v_{j+1}^{-}(t)-v_{j}^{-}(t)}{h}\right)^{2}-\left(\dfrac{v_{1}^{-}(t)}{h}\right)^{2}\\ 
&&\ \ \ \ \ \ \ \ \ \ \ \ \ \ \ \ \ \ \ \ \ \ \ \ \ \ \ +pM_{h}^{p-1}\sum\limits_{j=1}^{N}\left(v_{j}^{-}\right)^{2}(t) +b_{\infty}q\left(\dfrac{M_{h}}{h}\right)^{q-1}\sum\limits_{j=1}^{N} \left| \dfrac{v_{j+1}(t)-v_{j-1}(t)}{2h} \right|v_{j}^{-}(t)
\end{eqnarray*}
Then we get
\begin{eqnarray}
	\dfrac{1}{2}\dfrac{d}{dt}\sum\limits_{j=1}^{N}\left(v_{j}^{-}(t)\right)^{2}&\leq& -\sum\limits_{j=1}^{N}\dfrac{v_{j+1}^{+}(t)v_{j}^{-}(t)+v_{j}^{+}(t)v_{j+1}^{-}(t)}{h^{2}}+pM_{h}^{p-1}\sum\limits_{j=1}^{N}\left(v_{j}^{-}(t)\right)^{2}\\ \nonumber  &&+b_{\infty}q\left(\dfrac{M_{h}}{h}\right)^{q-1}\sum\limits_{j=1}^{N} \left| \dfrac{v_{j+1}(t)-v_{j-1}(t)}{2h} \right|v_{j}^{-}(t)-\dfrac{1}{h^{2}}\left( v_{1}^{-}(t)\right) ^{2}
	\label{6etoile}	
\end{eqnarray}
But 
\begin{eqnarray*}
&&\sum\limits_{j=1}^{N} \left| \dfrac{v_{j+1}(t)-v_{j-1}(t)}{2h} \right|v_{j}^{-}(t)\\
&\leq & \sum\limits_{j=1}^{N} \left| \dfrac{v_{j+1}^{+}(t)-v_{j-1}^{+}(t)}{2h} \right|v_{j}^{-}(t)+\sum\limits_{j=1}^{N} \left| \dfrac{v_{j+1}^{-}(t)-v_{j-1}^{-}(t)}{2h} \right|v_{j}^{-}(t)\\
&\leq& \dfrac{1}{2h}\left(\sum\limits_{j=1}^{N}v_{j+1}^{+}(t)v_{j}^{-}(t)+\sum\limits_{j=1}^{N}v_{j-1}^{+}(t)v_{j}^{-} (t)\right)+ \sum\limits_{j=1}^{N} \left| \dfrac{v_{j+1}^{-}(t)-v_{j}^{-}(t)}{2h} \right|v_{j}^{-}(t)+\sum\limits_{j=1}^{N} \left| \dfrac{v_{j}^{-}(t)-v_{j-1}^{-}(t)}{2h} \right|v_{j}^{-}(t)\\
&=&\dfrac{1}{2h}\left(\sum\limits_{j=1}^{N}v_{j+1}^{+}(t)v_{j}^{-}(t)+\sum\limits_{j=1}^{N}v_{j}^{+}(t)v_{j+1}^{-}(t) \right)+ \sum\limits_{j=1}^{N} \left| \dfrac{\left( v_{j+1}^{-}\right)^{2}(t) -\left( v_{j}^{-}\right) ^{2}(t)}{2h} \right|+\dfrac{1}{2}\left( \dfrac{v_{1}^{-}(t)}{h}\right)^{2} \\
&\leq& \dfrac{1}{2h}\left(\sum\limits_{j=1}^{N}v_{j+1}^{+}(t)v_{j}^{-}(t)+\sum\limits_{j=1}^{N}v_{j}^{+}(t)v_{j+1}^{-}(t) \right)+\dfrac{1}{h}\sum\limits_{j=1}^{N}\left(v_{j}^{-}\right)^{2}(t)- \dfrac{1}{2h}\left(v_{1}^{-}\right)^{2}(t)+\dfrac{1}{2}\left( \dfrac{v_{1}^{-}(t)}{h}\right)^{2}.
\end{eqnarray*}
Then \eqref{6etoile} implies
\begin{eqnarray*}
\dfrac{1}{2}\dfrac{d}{dt}\sum\limits_{j=1}^{N}\left(v_{j}^{-}(t)\right)^{2}&\leq& \left( -\dfrac{1}{h^{2}} +b_{\infty}q\left(\dfrac{M_{h}}{h}\right)^{q-1}\dfrac{1}{2h}\right) \sum\limits_{j=1}^{N}\left(v_{j+1}^{+}(t)v_{j}^{-}(t)+v_{j}^{+}(t)v_{j+1}^{-}(t)\right)\\
&&\ \ \ \ \ \ \ \ +\left(\dfrac{b_{\infty}q}{h}\left(\dfrac{M_{h}}{h}\right)^{q-1}+pM_{h}^{p-1}\right)\sum\limits_{j=1}^{N}\left(v_{j}^{-}(t)\right)^{2}\\
 &&\ \ \ \ \ \ \ \ \ \ \ \ +\left(- \dfrac{1}{h^{2}}+\dfrac{b_{\infty}q}{2h^{2}}\left(\dfrac{M_{h}}{h}\right)^{q-1}-\dfrac{b_{\infty}q}{2h}\left(\dfrac{M_{h}}{h}\right)^{q-1}\right) \left(v_{1}^{-}(t) \right)^{2} 
 \end{eqnarray*}
 and hence
 	\begin{eqnarray*}
 \dfrac{d}{dt}\sum\limits_{j=1}^{N}\left(v_{j}^{-}(t)\right)^{2}		
 &\leq& 2\left( -\dfrac{1}{h^{2}}+b_{\infty}q\left(\dfrac{M_{h}}{h}\right)^{q-1}\dfrac{1}{2h}\right) \sum\limits_{j=1}^{N}\left(v_{j+1}^{+}(t)v_{j}^{-}(t)+v_{j}^{+}(t)v_{j+1}^{-}(t)\right)\\
 &&\ \ \ \ \ \ \ \ +2\left(\dfrac{b_{\infty}q}{h}\left(\dfrac{M_{h}}{h}\right)^{q-1}+\dfrac{b_{\infty}q}{2h^{2}}\left(\dfrac{M_{h}}{h}\right)^{q-1}+pM_{h}^{p-1}\right)\sum\limits_{j=1}^{N}\left(v_{j}^{-}(t)\right)^{2}
\end{eqnarray*}
Let $M_{0}:=\max\limits_{j=1,...,N+1}u_{j}^{0}$ and we suppose that $h\ll \left(\dfrac{2}{b_{\infty}q}M_{0}^{-q+1}\right)^{\frac{1}{2-q}}$ is sufficiently small.\\
Since $u_{j}$ is a continuous function in $[0,T_h]$ then there exists $0<t_{h}\leq T_h$ such that $$h<\left(\dfrac{2}{b_{\infty}q}M_{h}^{-q+1}\right)^{\frac{1}{2-q}}.$$\\
Then 
$-\dfrac{1}{h^{2}}+b_{\infty}q\left(\dfrac{M_{h}}{h}\right)^{q-1}\dfrac{1}{2h}<0.$
Hence we get
\begin{equation*}
\dfrac{d}{dt}\left\|V_{h}^{-}(t) \right\|_{2}^{2}\leq 2\left(\dfrac{b_{\infty}q}{h}\left(\dfrac{M_{h}}{h}\right)^{q-1}+\dfrac{b_{\infty}q}{2h^{2}}\left(\dfrac{M_{h}}{h}\right)^{q-1}+pM_{h}^{p-1}\right)\left\|V_{h}^{-}(t) \right\|_{2}^{2}	
\end{equation*}
Now, integrating over $[0,t]$, using the Gronwall inegality and the hypothesis\\ $\left\|V_{h} (0)\right\|:=\left\|\dfrac{dU_{h}}{dt} (0)\right\|\geq 0 $, 	we get that $\left\|V_{h}^{-}(t) \right\|_{2}=0$ and this proves that $V_{h}\geq 0$ in $[0,t_{h}]$.\\
Let $t^{*}_{h}$ the greatest value of $t_{h}\in [0,T_h]$ such that
$\dfrac{dU_{h}}{dt}(t)\geq 0,\  \forall t\in[0,t_{h}^{*}].$
 In the next step, we shall prove that $t_{h}^{*}=T_h$. We suppose that $t_{h}^{*}<T_h$, we have $\dfrac{dU_{h}}{dt}(t_{h}^{*})\geq 0$. We use the same argument for the initial data $u(t_{h}^{*})$ we get
\begin{equation}
\exists \tau>0 \ \text{such that} \ [t_{h}^{*},t_{h}^{*}+\tau]\subset [0,T_h]\ \text{and }\ \dfrac{dU_{h}}{dt}\geq 0 \ \text{in }\ [t_{h}^{*},t_{h}^{*}+\tau]
\end{equation} 
which contradicts the definition of $t_{h}^{*}$.
			\end{proof}
			\begin{lem}
			 For all $t\in[0,T^*_{h})$ we define the function
			 \begin{equation*}
			 J(t)=\dfrac{1}{2}\sum\limits_{j=1}^{N+1}{\dfrac{(u_{j}(t)-u_{j-1}(t))^{2}}{h}}-\dfrac{1}{p+1}\sum\limits_{j=1}^{N+1}{hu^{p+1}_{j}(t)}.
			 \end{equation*}
			 Then $J$ is a nonincreasing function.
			\end{lem}
			\begin{proof}
			We multiply the first equation of \eqref{Eh} by $h\dfrac{du_{j}}{dt}(t)$ we get
			\begin{eqnarray*}
			&&h\left(\dfrac{du_{j}}{dt}(t)\right)^{2}\\
			&=& 
		    h\dfrac{u_{j+1}(t)-2u_{j}(t)+u_{j-1}(t)}{h^{2}}\dfrac{du_{j}}{dt}(t)+h(u_{j}(t))^{p}\dfrac{du_{j}}{dt}(t)-b_jh\left|\dfrac{u_{j+1}(t)-u_{j-1}(t)}{2h}\right|^{q}\dfrac{du_{j}}{dt}(t)  \\ 
		   	&=&	\dfrac{u_{j+1}(t)-2u_{j}(t)+u_{j-1}(t)}{h}\dfrac{du_{j}}{dt}(t)+\dfrac{h}{p+1}\dfrac{d}{dt}(u_{j}^{p+1}(t))-b_jh\left|\dfrac{u_{j+1}(t)-u_{j-1}(t)}{2h}\right|^{q}\dfrac{du_{j}}{dt} (t)\\
			\end{eqnarray*}
			We sum for $j=1,...,N$ we get
				\begin{eqnarray}
				\sum\limits_{j=1}^{N}{h\left(\dfrac{du_{j}}{dt}(t)\right)^{2}}&=& 
				\sum\limits_{j=1}^{N}{\dfrac{u_{j+1}(t)-u_{j}(t)}{h}\dfrac{du_{j}}{dt}(t)}-\sum\limits_{j=1}^{N}{\dfrac{u_{j}(t)-u_{j-1}(t)}{h}\dfrac{du_{j}}{dt}}(t) \label{1}\\
				&&\ \ +\sum\limits_{j=1}^{N}{\dfrac{h}{p+1}\dfrac{d}{dt}(u_{j}^{p+1})(t)}-\sum\limits_{j=1}^{N}{b_jh\left|\dfrac{u_{j+1}(t)-u_{j-1}(t)}{2h}\right|^{q}\dfrac{du_{j}}{dt}(t)}. \nonumber	
				\end{eqnarray}
				But
				\begin{eqnarray*}
			&&	\sum\limits_{j=1}^{N}{\dfrac{u_{j+1}(t)-u_{j}(t)}{h}\dfrac{du_{j}}{dt}(t)}-\sum\limits_{j=1}^{N}{\dfrac{u_{j}(t)-u_{j-1}(t)}{h}\dfrac{du_{j}}{dt}(t)}\\
				&=& \sum\limits_{j=2}^{N+1}{\dfrac{u_{j}(t)-u_{j-1}(t)}{h}\dfrac{du_{j-1}}{dt}(t)}-\sum\limits_{j=1}^{N}{\dfrac{u_{j}(t)-u_{j-1}(t)}{h}\dfrac{du_{j}}{dt}(t)} \\
				&=& -\sum\limits_{j=1}^{N}{\dfrac{u_{j}(t)-u_{j-1}(t)}{h}\dfrac{d}{dt}\left(u_{j}(t)-u_{j-1}(t)\right)}-\dfrac{u_{N}(t)}{h}\dfrac{du_{N}}{dt}(t)\\
				&=& -\dfrac{1}{2}\sum\limits_{j=1}^{N}{\dfrac{1}{h}\dfrac{d}{dt}(u_{j}(t)-u_{j-1}(t))^{2}}-\dfrac{1}{2h}\dfrac{d}{dt}(u^2_{N}(t))\\
				&=& -\dfrac{1}{2}\sum\limits_{j=1}^{N+1}{\dfrac{1}{h}\dfrac{d}{dt}(u_{j}(t)-u_{j-1}(t))^{2}}
				\end{eqnarray*}	
				Then \eqref{1} implies
			\begin{eqnarray*}
		&&	\sum\limits_{j=1}^{N}{h\left(\dfrac{du_{j}}{dt}(t)\right)^{2}}\\
		&=& -\dfrac{1}{2}\sum\limits_{j=1}^{N+1}{\dfrac{1}{h}\dfrac{d}{dt}(u_{j}(t)-u_{j-1}(t))^{2}}	 +\sum\limits_{j=1}^{N+1}{\dfrac{h}{p+1}\dfrac{d}{dt}(u_{j}^{p+1})(t)}-\sum\limits_{j=1}^{N}{b_jh\left|\dfrac{u_{j+1}(t)-u_{j-1}(t)}{2h}\right|^{q}\dfrac{du_{j}}{dt}(t)}\\
			&=& -\dfrac{d}{dt}J(t)-\sum\limits_{j=1}^{N}{hb_j\left|\dfrac{u_{j+1}(t)-u_{j-1}(t)}{2h}\right|^{q}\dfrac{du_{j}}{dt}(t)}
			\end{eqnarray*}
			which implies that: for all $t\in [0,T^*_h),$
			\begin{equation*}
			\dfrac{d}{dt}J(t)=-\left( \sum\limits_{j=1}^{N}{h\left(\dfrac{du_{j}}{dt}(t)\right)^{2}}+\sum\limits_{j=1}^{N}{hb_j\left|\dfrac{u_{j+1}(t)-u_{j-1}(t)}{2h}\right|^{q}\dfrac{du_{j}}{dt}}(t)\right) .
			\end{equation*}
			Using lemma \ref{lemcroissance} we can deduce that $	\dfrac{d}{dt}J(t)\leq 0$ which implies that $J$ is a nonincreasing function.
			\end{proof}
			In the next theorem and under some assumptions, we show that the semidiscrete solution blows up in a finite time.
			\begin{th1}
				Let $U_{h}$ be the nonnegative solution of \eqref{Eh} and we suppose that $J(0)<0$.\\
				We suppose also that:\\ $b_{\infty}<\dfrac{p-1}{\left(2^p(p+1)\right)^{\frac{1}{p+1}}}=\dfrac{p-1}{2}\bigg(\dfrac{2}{p+1}\bigg)^{1/p+1}$ if $q=\dfrac{2p}{p+1}$ and $\left\|U_{h}(0) \right\|_{p+1}>\left(\dfrac{c(p+1)}{p-1}\right)^{\frac{1}{\beta}}$ if $q<\dfrac{2p}{p+1}$ where $c=b_{\infty}\left(\dfrac{2}{p+1}\right)^{\frac{q}{2}}$ and $\beta=p-\dfrac{q(p+1)}{2}$.\\
				 Then $U_{h}$ achieves blow up in a finite time $T^*_{h}$.
			\end{th1}
			\begin{proof} We suppose that $T^*_{h}=+\infty$.
		We multiply the first equation of \eqref{Eh} by $hu_{j}$ and we sum for $j=1...,N$ we get
		\begin{eqnarray}
	\sum\limits_{j=1}^{N}{hu_{j}(t)\dfrac{du_{j}(t)}{dt}}&=&
	\sum\limits_{j=1}^{N}{hu_{j}(t)\dfrac{u_{j+1}(t)-2u_{j}(t)+u_{j-1}(t)}{h^{2}}} 
	 +\sum\limits_{j=1}^{N}{hu_{j}^{p+1}(t)}\\ \nonumber  &&\ \ -\sum\limits_{j=1}^{N}{b_jhu_{j}(t)\left|\dfrac{u_{j+1}(t)-u_{j-1}(t)}{2h}\right|^{q}}.	
	 \label{2}
		\end{eqnarray}
		But
			\begin{eqnarray*}
			 	\sum\limits_{j=1}^{N}{hu_{j}(t)\dfrac{u_{j+1}(t)-2u_{j}(t)+u_{j-1}(t)}{h^{2}}}
				&=&\sum\limits_{j=1}^{N}{\dfrac{u_{j+1}(t)-u_{j}(t)}{h}u_{j}(t)}-\sum\limits_{j=1}^{N}{\dfrac{u_{j}(t)-u_{j-1}(t)}{h}u_{j}(t)} \\
				&=& \sum\limits_{j=2}^{N+1}{\dfrac{u_{j}(t)-u_{j-1}(t)}{h}u_{j-1}(t)}-\sum\limits_{j=1}^{N}{\dfrac{u_{j}(t)-u_{j-1}(t)}{h}u_{j}(t)} \\
				&=& -\sum\limits_{j=1}^{N}{\dfrac{(u_{j}(t)-u_{j-1}(t))^{2}}{h}}-\dfrac{u^2_{N}(t)}{h}\\
				&=& -\sum\limits_{j=1}^{N+1}{\dfrac{(u_{j}(t)-u_{j-1}(t))^{2}}{h}}.
			\end{eqnarray*}	
			Then \eqref{2} implies
				\begin{equation*}
					\sum\limits_{j=1}^{N}{hu_{j}(t)\dfrac{du_{j}(t)}{dt}}=
					 -\sum\limits_{j=1}^{N+1}{\dfrac{(u_{j}(t)-u_{j-1}(t))^{2}}{h}}+\sum\limits_{j=1}^{N+1}{h(u_{j}(t))^{p+1}}-\sum\limits_{j=1}^{N}{hb_ju_{j}(t)\left|\dfrac{u_{j+1}(t)-u_{j-1}(t)}{2h}\right|^{q}}.	
				\end{equation*}
				which implies that
					\begin{eqnarray*}
					\dfrac{1}{2}\dfrac{d}{dt}\sum\limits_{j=1}^{N+1}{hu^2_{j}(t)}
				&=&
					-2\left(\dfrac{1}{2} \sum\limits_{j=1}^{N+1}{\dfrac{(u_{j}(t)-u_{j-1}(t))^{2}}{h}}-\dfrac{1}{p+1}\sum\limits_{j=1}^{N+1}{hu_{j}^{p+1}(t)}\right)
					\\ &&\ \ \ \ \ \ +\left(1-\dfrac{2}{p+1}\right)\sum\limits_{j=1}^{N+1}{hu_{j}^{p+1}(t)}-\sum\limits_{j=1}^{N}{hb_ju_{j}(t)\left|\dfrac{u_{j+1}(t)-u_{j-1}(t)}{2h}\right|^{q}}\\	
					&=& -2J(t)+\dfrac{p-1}{p+1}\sum\limits_{j=1}^{N+1}{hu_{j}^{p+1}(t)}-\sum\limits_{j=1}^{N}{hb_ju_{j}(t)\left|\dfrac{u_{j+1}(t)-u_{j-1}(t)}{2h}\right|^{q}}.
					\end{eqnarray*}				
					Then 
						\begin{equation*}
						\dfrac{d}{dt}\sum\limits_{j=1}^{N+1}{hu_{j}^{2}(t)}=	-4J(t)+\dfrac{2(p-1)}{p+1}\sum\limits_{j=1}^{N+1}{hu_{j}^{p+1}(t)}-2\sum\limits_{j=1}^{N}{hb_ju_{j}(t)\left|\dfrac{u_{j+1}(t)-u_{j-1}(t)}{2h}\right|^{q}}
						\end{equation*}
			implies that
			\begin{equation*}
			\dfrac{d}{dt}\left\| U_{h}(t)\right\| _{2}^{2}=	-4J(t)+\dfrac{2(p-1)}{p+1}	\left\| U_{h}(t)\right\|^{p+1} _{p+1}-2\sum\limits_{j=1}^{N}{hb_ju_{j}(t)\left|\dfrac{u_{j+1}(t)-u_{j-1}(t)}{2h}\right|^{q}} \end{equation*}
			but $J(t)\leq 0$ then
			\begin{equation}
			\dfrac{d}{dt}\left\| U_{h}(t)\right\| _{2}^{2}	\geq 	\dfrac{2(p-1)}{p+1}	\left\| U_{h}(t)\right\|^{p+1} _{p+1}-2\sum\limits_{j=1}^{N}{hb_ju_{j}(t)\left|\dfrac{u_{j+1}(t)-u_{j-1}(t)}{2h}\right|^{q}}
			\label{33}
			\end{equation}		
			Now using the discret inegality of H\"{o}lder we get
			\begin{eqnarray}
		     \sum\limits_{j=1}^{N}{hb_ju_{j}(t)\left|\dfrac{u_{j+1}(t)-u_{j-1}(t)}{2h}\right|^{q}}
			&\leq&b_{\infty} \left( \sum\limits_{j=1}^{N}{h^{p+1}u_{j}^{p+1}(t)}\right)^{\frac{1}{p+1}} \left(\sum\limits_{j=1}^{N}
			{\left(\left|\dfrac{u_{j+1}(t)-u_{j-1}(t)}{2h}\right|^{q}\right) ^{\frac{p+1}{p}}}\right) ^{\frac{p}{p+1}} \nonumber\\
			&=&b_{\infty} h^{\frac{p}{p+1}}\left( \sum\limits_{j=1}^{N}{hu_{j}^{p+1}(t)}\right)^{\frac{1}{p+1}} \left(\sum\limits_{j=1}^{N}
			{\left(\left|\dfrac{u_{j+1}(t)-u_{j-1}(t)}{2h}\right|^{q}\right) ^{\frac{p+1}{p}}}\right) ^{\frac{p}{p+1}} \nonumber\\
			&=& b_{\infty}h^{\frac{p}{p+1}} 	\left\| U_{h}(t)\right\|_{p+1} \left(\sum\limits_{j=1}^{N}
			{\left|\dfrac{u_{j+1}(t)-u_{j-1}(t)}{2h}\right|^{\frac{q(p+1)}{p}}}\right) ^{\frac{p}{p+1}}
			\label{3}
			\end{eqnarray}	
			On an other hand we have
			\begin{equation*}
			\left(\sum\limits_{j=1}^{N}
			{\left|\dfrac{u_{j+1}(t)-u_{j-1}(t)}{2h}\right|^{\frac{q(p+1)}{p}}}\right) ^{\frac{p}{p+1}} = 
			\left(\sum\limits_{j=1}^{N}
			{\left|\dfrac{u_{j+1}(t)-u_{j}(t)}{2h}+\dfrac{u_{j}(t)-u_{j-1}(t)}{2h}\right|^{\frac{q(p+1)}{p}}}\right) ^{\frac{p}{p+1}}.
			\end{equation*}	
			We use that $\left( \dfrac{\alpha+\beta}{2}\right) ^{n}\leq \dfrac{\alpha^{n}+\beta^{n}}{2}$, we get
			\begin{eqnarray}
		   &&\left(\sum\limits_{j=1}^{N}
			{\left|\dfrac{u_{j+1}(t)-u_{j-1}(t)}{2h}\right|^{\frac{q(p+1)}{p}}}\right) ^{\frac{p}{p+1}} \nonumber\\
			&\leq&  \left(\dfrac{\sum\limits_{j=1}^{N}
			{\left|\dfrac{u_{j+1}(t)-u_{j}(t)}{h}\right| ^{\frac{q(p+1)}{p}}}+\sum\limits_{j=1}^{N} {\left|\dfrac{u_{j}(t)-u_{j-1}(t)}{h}\right| ^{\frac{q(p+1)}{p}}}}{2}\right)^{\frac{p}{p+1}} \nonumber\\
			&\leq&  \dfrac{1}{2} \left(\sum\limits_{j=1}^{N}
			{\left|\dfrac{u_{j+1}(t)-u_{j}(t)}{h}\right| ^{\frac{q(p+1)}{p}}}\right)^{\frac{p}{p+1}}+ \dfrac{1}{2} \left(\sum\limits_{j=1}^{N} {\left|\dfrac{u_{j}(t)-u_{j-1}(t)}{h}\right| ^{\frac{q(p+1)}{p}}}\right)^{\frac{p}{p+1}} \nonumber\\ 
			&=&  \dfrac{1}{2} \left(\sum\limits_{j=1}^{N+1}
			{\left|\dfrac{u_{j}(t)-u_{j-1}(t)}{h}\right|^{\frac{q(p+1)}{p}}}\right)^{\frac{p}{p+1}}+ \dfrac{1}{2} \left(\sum\limits_{j=1}^{N+1} {\left|\dfrac{u_{j}(t)-u_{j-1}(t)}{h}\right| ^{\frac{q(p+1)}{p}}}\right)^{\frac{p}{p+1}} \nonumber\\
			&\leq &  \left(\sum\limits_{j=1}^{N+1}
			{\left|\dfrac{u_{j}(t)-u_{j-1}(t)}{h}\right|^{\frac{q(p+1)}{p}}}\right)^{\frac{p}{p+1}} \nonumber\\
			&=&\dfrac{1}{h^{\frac{p}{p+1}}} \left(\sum\limits_{j=1}^{N+1}
			{h\left|\dfrac{u_{j}(t)-u_{j-1}(t)}{h}\right|^{\frac{q(p+1)}{p}}}\right)^{\frac{p}{p+1}} \nonumber\\
			&=& \dfrac{1}{h^{\frac{p}{p+1}}}\left\|(\delta^{-}_{x}U_{h}(t))^{q} \right\|_{\frac{p+1}{p}} 
			\label{3etoile}	
			\end{eqnarray}	
			with
			\begin{eqnarray*}
			\left\|(\delta^{-}_{x}U_{h}(t))^{q} \right\|_{\frac{p+1}{p}}&=& \left(\sum\limits_{j=1}^{N+1}
			{h\left|\dfrac{u_{j}(t)-u_{j-1}(t)}{h}\right|^{\frac{q(p+1)}{p}}}\right)^{\frac{p}{p+1}}\\
			&=&\left[ \left(\sum\limits_{j=1}^{N+1}
			{h\left|\dfrac{u_{j}(t)-u_{j-1}(t)}{h}\right|^{\frac{q(p+1)}{p}}}\right)^{\frac{p}{q(p+1)}}\right] ^{q}\\
			&=& 	\left\|\delta^{-}_{x}U_{h}(t) \right\|^{q}_{\frac{q(p+1)}{p}}	
				\end{eqnarray*}	
			Then from \eqref{3etoile} and using that $q\leq \dfrac{2p}{p+1}$ we get
				\begin{eqnarray}
					\left(\sum\limits_{j=1}^{N}
					{\left|\dfrac{u_{j+1}(t)-u_{j-1}(t)}{2h}\right|^{\frac{q(p+1)}{p}}}\right) ^{\frac{p}{p+1}}&\leq& \dfrac{1}{h^{\frac{p}{p+1}}}\left\|\delta^{-}_{x}U_{h}(t) \right\|^{q}_{\frac{q(p+1)}{p}}	 \nonumber\\
					&\leq& \dfrac{1}{h^{\frac{p}{p+1}}}\left\|\delta^{-}_{x}U_{h}(t) \right\|^{q}_{2}
					\label{4etoile}
			\end{eqnarray}	
			Therefore \eqref{3} and \eqref{4etoile} implies
			\begin{eqnarray*}	
			 \sum\limits_{j=1}^{N}{hb_ju_{j}(t)\left|\dfrac{u_{j+1}(t)-u_{j-1}(t)}{2h}\right|^{q}}&\leq& b_{\infty}h^{\frac{p}{p+1}}\left\|U_{h} (t)\right\|_{p+1} \dfrac{1}{h^{\frac{p}{p+1}}}\left\|\delta^{-}_{x}U_{h}(t) \right\|^{q}_{2}\\
			 &=&b_{\infty} \left\|U_{h}(t) \right\|_{p+1} \left\|\delta^{-}_{x}U_{h}(t) \right\|^{q}_{2}.
			\end{eqnarray*} 
			Now using that	$J\leq 0$ we get
			\begin{eqnarray*}
			\dfrac{1}{2}\sum\limits_{j=1}^{N+1}{h\left( \dfrac{u_{j}(t)-u_{j-1}(t)}{h}\right) ^{2}}-\dfrac{1}{p+1}\left\|U_{h}(t) \right\|^{p+1}_{p+1}\leq 0
			&\Rightarrow &\left\|\delta^{-}_{x}U_{h} (t)\right\|^{2}_{2}\leq \dfrac{2}{p+1}\left\|U_{h} (t)\right\|^{p+1}_{p+1}\\
			&\Rightarrow& \left\|\delta^{-}_{x}U_{h}(t) \right\|^{q}_{2}\leq \left( \dfrac{2}{p+1}\right) ^{\frac{q}{2}}\left\|U_{h}(t) \right\|^{\frac{q(p+1)}{2}}_{p+1}\\
			\end{eqnarray*}
		So
		\begin{eqnarray}
		\sum\limits_{j=1}^{N}{hb_ju_{j}(t)\left|\dfrac{u_{j+1}(t)-u_{j-1}(t)}{2h}\right|^{q}} &\leq &b_{\infty} \left\|U_{h} (t)\right\|_{p+1} \left( \dfrac{2}{p+1}\right) ^{\frac{q}{2}}\left\|U_{h}(t) \right\|^{\frac{q(p+1)}{2}}_{p+1} \nonumber\\
		&=& c \left\|U_{h} (t)\right\|^{p+1-\beta}_{p+1}
		\label{5etoile}
		\end{eqnarray}
	with $c=b_{\infty}\left( \dfrac{2}{p+1}\right) ^{\frac{q}{2}}>0$ and $\beta=p-\dfrac{q(p+1)}{2}\geq 0$.\\
		Finally \eqref{33} and \eqref{5etoile} gives
		\begin{eqnarray*}
		\dfrac{d}{dt}\left\|U_{h} (t)\right\|^{2}_{2}&\geq & \dfrac{2(p-1)}{p+1}\left\|U_{h}(t) \right\|^{p+1}_{p+1}-2c\left\|U_{h}(t) \right\|^{p+1-\beta}_{p+1}\\
		&=& 2 \left\|U_{h}(t) \right\|^{p+1}_{p+1}\left( \dfrac{p-1}{p+1}-c \left\|U_{h}(t) \right\|^{-\beta}_{p+1}\right) 	
		\end{eqnarray*}
		Let $F=\left\|U_{h} \right\|^{2}_{2}$, as $p+1\geq 2$ then $F\leq \left\|U_{h} \right\|^{2}_{p+1}$ which implies that $\left\|U_{h} \right\|^{p+1}_{p+1}\geq F^{\frac{p+1}{2}}.$
		Then for all $t\geq 0$
			\begin{equation*}
				F'(t)\geq 2 F^{\frac{p+1}{2}}(t)\left( \dfrac{p-1}{p+1}-c \left\|U_{h}(t) \right\|^{-\beta}_{p+1}\right) 
					\end{equation*}
		As $\dfrac{dU_{h}(t)}{dt}\geq 0$ then $U_{h}(t)\geq U_{h}(0)$ for all $t\geq 0$, so
		\begin{equation}
		F'(t)\geq 2 F^{\frac{p+1}{2}}(t)\left( \dfrac{p-1}{p+1}-c \left\|U_{h}(0) \right\|^{-\beta}_{p+1}\right) 
		\label{F}
		\end{equation}	
		Let $k:=2\left( \dfrac{p-1}{p+1}-c \left\|U_{h}(0) \right\|^{-\beta}_{p+1}\right)$.\\
		Note that if $q< \dfrac{2p}{p+1}$ and for a large initial data satisfying $\left\|U_{h}(0) \right\|_{p+1}>\left(\dfrac{c(p+1)}{p-1}\right)^{\frac{1}{\beta}}$, we have $k>0$.\\
		If $q= \dfrac{2p}{p+1}$ and $0<b_{\infty}<\dfrac{p-1}{2}\bigg(\dfrac{2}{p+1}\bigg)^{1/p+1}$, then $k>0$.\\
		In the two cases we have
		\begin{eqnarray*}
	 F'(t)\geq k  F^{\frac{p+1}{2}}(t)
		&\Rightarrow&  F'(t)F^{-\frac{p+1}{2}}(t)\geq k.\\
		&\Rightarrow&  \dfrac{2}{1-p}F^{\frac{1-p}{2}}(t)\geq kt-\dfrac{2}{1-p}F^{\frac{1-p}{2}}(0).
		\end{eqnarray*}
		which is impossible for all $t\geq 0$ sufficiently large and $F(0)>0$, this contradiction shows that $T^*_{h}<\infty$ and hence we get blow up of the semidiscrete solution.		
		\end{proof}
	\begin{rem}
		We note that the upper-bound of $b_\infty$ in the semidiscrete problem is the same proved in the theoretical case in \cite{alfonsi}, when $b$ is a constant.
	\end{rem}
		\section{Convergence of the semidiscrete scheme}
		\noindent In this section, we prove the convergence of the semidiscrete scheme \eqref{Eh}.\\
		Let $u$ be the solution of \eqref{E}. For each $h$, we denote $u_{h}(t)=\left(0,u(x_{1},t),...,u(x_{N},t),0\right)'$. The next theorem establishes that for each fixed time interval $[0,T]$ where $u$ is defined, the solution of the semidiscrete problem \eqref{Eh} approximates $u$ as $h\longrightarrow 0.$
		\begin{th1}
			Let $\tilde{T}^{*}_{h}=\min(T^*_{h},T^{*})$ and $u_{h}(t)$ be the exact solution of \eqref{E}.We suppose that $u_{h}\in C^{4}((0,T^{*}),\mathbb{R}^{N+2})$, $J(0)<0$ and we assume that the initial condition $U_{h}^{0}$ satisfies 
			\begin{equation}
			\left\|U_{h}^{0}-u_{h}(0) \right\|_{\infty} =o(1) \text{ as } h\longrightarrow 0.
			\label{7etoile}
			\end{equation}
			Then for $h$ sufficiently small, problem \eqref{Eh} has a unique solution $U_{h}\in C^{1}\left([0,T^*_{h}), \mathbb{R}^{N+2}\right)$ such that
			\begin{equation*}
			\max\limits_{t\in [0,T]}\left\|U_{h}(t)-u_{h}(t) \right\|_{\infty} =O\left(\left\|U_{h}^{0}-u_{h}(0) \right\|_{\infty}+h^{2}\right) \text{ as } h\longrightarrow 0 \ \text{for all} \ T<\tilde{T}^{*}_{h}.
			\label{8etoile}
			\end{equation*}
		\end{th1}
		\begin{proof}
			Let $T<\tilde{T}^{*}_{h}$, for all $t\in [0,T)$, we denote $e_{h}(t)=U_{h}(t)-u_{h}(t)$ the error of discretization. Using \eqref{E}, \eqref{Eh} and Taylor's expansion we get that for $j=1,...,N$
			\begin{eqnarray}
			\dfrac{de_{j}}{dt}(t)-\delta^{2}_{x}e_{j}(t)=U_{j}^{p}(t)&-& u^{p}(x_{j},t)-b_j \left(\left| \dfrac{u_{j+1}(t)-u_{j-1}(t)}{2h} \right|^{q}-\left| \dfrac{u(x_{j+1},t)-u(x_{j-1},t)}{2h}\right| ^{q}  \right) \nonumber\\
			&&+\dfrac{h^{2}}{24}\dfrac{\partial^{4}u}{\partial x^{4}}(\tilde{x_{j}},t)+\dfrac{h^{2}}{24}\dfrac{\partial^{4}u}{\partial x^{4}}(\tilde{\tilde{x_{j}}},t)+o(h^{2}) 
			\label{10etoile}
			\end{eqnarray}
			Now we denote $f(X)=X^{p}$ and we use the mean value theorem we get
			\begin{equation*}
			\left|U_{j}^{p}(t)-u^{p}(x_{j},t) \right|\leq \left|f'(z_{j}) \right| \left|e_{j}(t) \right|  
			\end{equation*}	
			where $z_{j}$ is an intermediate value between $U_{j}(t)$ and $u(x_{j},t)$. Since $u_{h}(t)$ and $U_{h}(t)$ are bounded for all $t\in [0,T]$, then we can suppose that there exists a positive constant $c_{1}$ such that $\left|f'(z_{j}) \right|\leq c_{1}$.\\
			On the other hand, we denote $g(X)=\left| X\right| ^{q}$ and we use the mean value theorem we get
			\begin{eqnarray*}
		&&	\left|\left| \dfrac{u_{j+1}(t)-u_{j-1}(t)}{2h} \right|^{q}-\left| \dfrac{u(x_{j+1},t)-u(x_{j-1},t)}{2h}\right| ^{q} \right|\\
		&\leq& \left|g'(\xi_{j}) \right| \left| \dfrac{u_{j+1}(t)-u_{j-1}(t)}{2h} - \dfrac{u(x_{j+1},t)-u(x_{j-1},t)}{2h} \right|  
			\end{eqnarray*}	
			where $\xi_{j}$ is an intermediate value between $\dfrac{u_{j+1}(t)-u_{j-1}(t)}{2h}$ and $\dfrac{u(x_{j+1},t)-u(x_{j-1},t)}{2h}$.\\
			Let prove that $\dfrac{u_{j+1}(t)-u_{j-1}(t)}{2h}$ and $\dfrac{u(x_{j+1},t)-u(x_{j-1},t)}{2h}$ are bounded.\\
			Using Taylor's formula we have
			\begin{equation*}
			\dfrac{u_{j+1}(t)-u_{j-1}(t)}{2h}=\dfrac{\partial u}{\partial x}(x_{j},t)+h\left(\dfrac{\partial^{2} u}{\partial x^{2}}(\tilde{x_{j}},t)+\dfrac{\partial^{2} u}{\partial x^{2}}(\tilde{\tilde{x_{j}}},t)\right)
			\end{equation*}
			But $u_{h}\in C^{4}((0,T^{*}),\mathbb{R}^{N+2})$ then $\dfrac{u(x_{j+1},t)-u(x_{j-1},t)}{2h}$ is bounded.\\
			On the other hand we have
			\begin{eqnarray*}
		    \left\|\delta_{x}U_{h}(t)\right\|_{2}^{2}&=& \sum\limits_{j=1}^{N}h\left|\dfrac{u_{j+1}(t)-u_{j-1}(t)}{2h}\right|^{2}\\
		    &=& \sum\limits_{j=1}^{N}\dfrac{1}{4h}\left|u_{j+1}(t)-u_{j}(t)+u_{j}(t)-u_{j-1}(t)\right|^{2}\\
		    &\leq& \dfrac{1}{2h}\sum\limits_{j=1}^{N}\left|u_{j+1}(t)-u_{j}(t)\right|^{2}+\dfrac{1}{2h}\sum\limits_{j=1}^{N}\left|u_{j}(t)-u_{j-1}(t)\right|^{2}\\
		    &\leq& \sum\limits_{j=1}^{N+1}\dfrac{(u_{j}(t)-u_{j-1}(t))^{2}}{h}\\
		    &=& 2J(t)+\dfrac{2}{p+1}\left\|U_{h}(t)\right\|^{p+1}_{p+1}
			\end{eqnarray*}
			Using that $J$ is nonincreasing and $J(0)<0$ then
			\begin{equation*}
			\left\|\delta_{x}U_{h}(t)\right\|^{2}_{2}\leq \dfrac{2}{p+1}\left\|U_{h}(t)\right\|^{p+1}_{p+1}<+\infty .
			\end{equation*}			 
			Finally, since $\dfrac{u_{j+1}(t)-u_{j-1}(t)}{2h}$ and $\dfrac{u(x_{j+1},t)-u(x_{j-1},t)}{2h}$ are bounded before blow up, then we can suppose that there exists a positive constant $c_{2}$ such that $\left|g'(\xi_{j}) \right|\leq c_{2}$.\\		
		Finally $e_{j}(t)$ satisfies
		\begin{equation*}
		\dfrac{de_{j}(t)}{dt}-\delta_{x}^{2}e_{j}(t)\leq c_{1}\left|e_{j}(t)\right|+b_{\infty}c_{2}\left|\delta_{x}e_{j}\right|+\dfrac{h^{2}}{24}\dfrac{\partial^{4}u}{\partial x^{4}}(\tilde{x_{j}},t)+\dfrac{h^{2}}{24}\dfrac{\partial^{4}u}{\partial x^{4}}(\tilde{\tilde{x_{j}}},t)+o(h^{2})
		\end{equation*}
			Let $R$ and $K$ be two positive constants such that 
			\begin{equation*}		
			\bigg{\vvvert{\dfrac{\partial ^{4}u}{\partial x^{4} }}\bigg\vvvert} \leq R \text{ and } K=o(1)+\dfrac{R}{12}
			\end{equation*}
			then \eqref{10etoile} implies 
			\begin{equation}
				\dfrac{de_{j}}{dt}(t)-\delta^{2}_{x}e_{j}(t)- c_{1}\left|e_{j}(t) \right|-b_{\infty}c_{2}\left|\delta_{x}e_{j}\right|-Kh^{2}\leq 0\\ 
				\label{negative}
			\end{equation}
			We consider now the function $W_{h}$ defined by 
			\begin{equation*}
			W_{j}(t)=\exp((c_{1}+1)t)\left( \left\|e(0) \right\|_{\infty} +Kh^{2}\right),\ \ \ 0\leq j\leq N+1,\ \ \ t\in [0,T).
			\end{equation*}
			But for all $0\leq j\leq N+1$ and $t\in (0,T)$, $W_{j}(t)$ satisfies
			\begin{eqnarray}
				&&	\dfrac{dW_{j}}{dt}(t)-\delta^{2}_{x}W_{j}(t)- c_{1}\left|W_{j}(t) \right|-b_{\infty}c_{2}\left|\delta_{x}W_{j}(t)\right|-Kh^{2}  \nonumber\\
				&=& (c_{1}+1)\exp((c_{1}+1)t)\left( \left\|e(0) \right\|_{\infty} +Kh^{2}\right)-c_{1}\exp((c_{1}+1)t)\left( \left\|e(0) \right\|_{\infty} +Kh^{2}\right)-Kh^{2} \nonumber\\
				&=&\exp((c_{1}+1)t)\left( \left\|e(0) \right\|_{\infty} +Kh^{2}\right)-Kh^{2} \nonumber \\
				&>& 0.
				\label{positif}
			\end{eqnarray}		
			And
			\begin{equation*}
			\left\lbrace 
				\begin{array}{lll}	
				W_{N+1}(t)=\exp((c_{1}+1)t)\left( \left\|e(0) \right\|_{\infty} +Kh^{2}\right)>0=e_{N+1}(t),\ \ \ t\in (0,T),\\
				W_{j}(0)= \left\|e(0) \right\|_{\infty}+Kh^{2}>e_{j}(0),\ \ \ 0\leq j\leq N+1. 
				 \end{array}
				 \right.
			\end{equation*}
			Next, we need the lemma below which	is another form of the maximum principle for semidiscrete equations called the comparaison lemma (proved in \cite{nabongo}, lemma 2.3).
				\begin{lem}
					Let  $U_{h}(t),V_{h}(t)\in C^{1}\left((0,T),\mathbb{R}^{N+2}\right) $ and $f\in C^{0}\left(\mathbb{R}\times\mathbb{R},\mathbb{R}\right) $ such that for $t\in (0,T)$
					\begin{equation*}
						\left\{ 
						\begin{array}{lllll}
							\dfrac{dv_{j}(t)}{dt}-\delta^{2}_{x}v_{j}(t)+f(v_{j}(t),t)< \dfrac{du_{j}(t)}{dt}-\delta^{2}_{x}u_{j}(t)+f(u_{j}(t),t), \ \ 1\leq j \leq N\\
							v_{0}(t)<u_{0}(t), v_{N+1}(t)<u_{N+1}(t),\ \ \\
							v_{j}^{0}<u_{j}^{0} \text{ for }1\leq j \leq N.
						\end{array}
						\right.\\
					\end{equation*} 	
					then we have $v_{i}(t)<u_{j}(t)$ for all $t\in (0,T)$ and $1\leq j\leq N.$
					\label{lemcomparaison}
				\end{lem}
		Using \eqref{negative}, \eqref{positif} and lemma \ref{lemcomparaison} we get that
			\begin{equation*}
			W_{j}(t)>e_{j}(t) \text{ for } t\in(0,T),\ \ 0\leq j\leq N+1. 
			\end{equation*}
			Using the same argument for $-e$ we also show that
			\begin{equation*}
			W_{j}(t)>-e_{j}(t) \text{ for } t\in(0,T),\ \ 0\leq j\leq N+1. 
			\end{equation*}
			which implies that 
			\begin{eqnarray*}
				&& \left| e_{j}(t)\right| < W_{j}(t)\ \ \ \ t\in(0,T^{*}_{h}),\ \ 0\leq j\leq N+1\\ 
				&\Rightarrow&	\left\|U_{h}(t)-u_{h}(t) \right\|_{\infty} \leq \exp((c_{1}+1)t)\left(	\left\|U_{h}^{0}-u_{h}(0) \right\|_{\infty}+Kh^{2}\right), \ \ \ t\in(0,T).		 
			\end{eqnarray*}
			And then
			\begin{equation*}
			\max\limits_{t\in [0,T)}\left\|U_{h}(t)-u_{h}(t) \right\|_{\infty} =O\left(	\left\|U_{h}^{0}-u_{h}(0) \right\|_{\infty}+h^{2}\right) \text{ as } h\longrightarrow 0.
			\end{equation*}
		\end{proof}
			\section{Numerical blow-up rate}
			In this section, we consider positive solution of \eqref{Eh}, we give the asymptotic behavior (blow up rate) of $U_{h}$ and we prove the following theorem:
			\begin{th1}
				Let $U_{h}$ be the numerical solution of \eqref{Eh}. Assume that $p>1$, $1\leq q \leq \dfrac{2p}{p+1}:$	
				If $q< \dfrac{2p}{p+1}$ we choose a large initial data.
				If $q= \dfrac{2p}{p+1}$ we choose $0<b_{\infty}<\dfrac{p-1}{2}\bigg(\dfrac{2}{p+1}\bigg)^{1/p+1}$.\\
				We suppose also that  $U_{h}$ blows up in finite time $T_{h}$. Then, there exists two positive constants $C_{1}$ and $C_{2}$ such that
				\begin{equation*}
				C_{1}\left(T^*_{h}-t\right)^{-\frac{1}{p-1}}\leq \max_{1\leq j\leq N}u_{j}(t)\leq 	C_{2}\left(T^*_{h}-t\right)^{-\frac{1}{p-1}}\ \ \ for \ all \ t\in (0, T^*_h)
				\end{equation*}
			\end{th1} 
			\begin{proof}
				For all $j=1,...,N$ and $t\in (0,T^*_{h})$ we have
				\begin{equation*}
				u'_{j}(t)=\dfrac{u_{j+1}(t)-2u_{j}(t)+u_{j-1}(t)}{h^{2}}+u_{j}^{p}(t)-b\left|\dfrac{u_{j+1}(t)-u_{j-1}(t)}{2h}\right|^{q}.
				\end{equation*}
				which implies that
				\begin{equation*}
				u'_{j}(t)\leq \dfrac{u_{j+1}(t)-2u_{j}(t)+u_{j-1}(t)}{h^{2}}+u_{j}^{p}(t).
				\end{equation*}
				We multiply by $u_{j}(t)\geq 0$ and we sum for $j=1,...,N$ we get
				\begin{equation}
				\sum\limits_{j=1}^{N}u'_{j}(t)u_{j}(t)\leq \dfrac{1}{h^{2}}\left(\sum\limits_{j=1}^{N}\left(u_{j+1}(t)-u_{j}(t)\right)u_{j}(t)+\sum\limits_{j=1}^{N}\left(u_{j-1}(t)-u_{j}(t)\right)u_{j}(t)\right)+\sum\limits_{j=1}^{N}u_{j}^{p+1}(t).
				\label{estimation1}
				\end{equation}
				But
				\begin{eqnarray*}
					&&\sum\limits_{j=1}^{N}\left(u_{j+1}(t)-u_{j}(t)\right)u_{j}(t)+\sum\limits_{j=1}^{N}\left(u_{j-1}(t)-u_{j}(t)\right)u_{j}(t)\\
					&=&\sum\limits_{j=1}^{N}\left(u_{j+1}(t)-u_{j}(t)\right)u_{j}(t)-\sum\limits_{j=1}^{N}\left(u_{j+1}(t)-u_{j}(t)\right)u_{j+1}(t)-u_{1}^{2}(t)\\
					&=& -\sum\limits_{j=1}^{N}\left(u_{j+1}(t)-u_{j}(t)\right)^{2}-u_{1}^{2}(t)\\
					&\leq & 0.
				\end{eqnarray*}
				Then \eqref{estimation1} implies 
				\begin{equation*}
				\sum\limits_{j=1}^{N}u'_{j}(t)u_{j}(t)\leq\sum\limits_{j=1}^{N}u_{j}^{p+1}(t).
				\end{equation*}
				And hence
				\begin{equation}
				\dfrac{1}{2}\dfrac{d}{dt}\left(\sum\limits_{j=1}^{N}u_{j}^{2}(t)\right)\leq\sum\limits_{j=1}^{N}\left((u_{j}(t))^{2}\right)^{\frac{p+1}{2}}\leq \left(\sum\limits_{j=1}^{N}(u_{j}(t))^{2}\right)^{\frac{p+1}{2}}.
				\label{estimation2}
				\end{equation}
				In fact, let $v_{j}=u_{j}^{2}$ and $r=\dfrac{p+1}{2}$, we have to show that
				\begin{equation*}
				\left(\sum\limits_{j=1}^{N}v_{j}^{r}\right)^{\frac{1}{r}}\leq \sum\limits_{j=1}^{N}v_{j}.
				\end{equation*}
				Let 
				\begin{equation*}
				A=\left(\sum\limits_{j=1}^{N}v_{j}^{r}\right)^{\frac{1}{r}},\ \ B=\sum\limits_{j=1}^{N}v_{j}\ \ \text{and}\ \ C=\max\limits_{1\leq j \leq N} v_{j}
				\end{equation*}
				Clearly we have $C\leq B$, then 
				\begin{equation*}
				v_{j}^{r}=v_{j}v_{j}^{r-1}\leq v_{j}C^{r-1}
				\end{equation*}
				which implies that
				\begin{equation*}
				\sum\limits_{j=1}^{N}v_{j}^{r}\leq C^{r-1}\sum\limits_{j=1}^{N}v_{j} \leq  B^{r-1}B = B
				\end{equation*}
				and finally 
				\begin{equation*}
				A\leq B
				\end{equation*}
				this proves \eqref{estimation2}.\\
				We define now $w(t)=\sum\limits_{j=1}^{N}u_{j}^{2}(t)$, then
				\eqref{estimation2} implies
				\begin{equation}
				w'(t)\leq 2w^{r}(t)
				\label{estimate}
				\end{equation}
				Integrating the above inequality between $t$ and $T^*_{h}$, we obtain
				\begin{equation*}
				\int\limits_{t}^{T^*_{h}}\dfrac{w'(s)}{w^{r}(s)}\leq 2(T^*_{h}-t).
				\end{equation*} 
				Changing variables and we use that $\lim\limits_{t\rightarrow T^*_{h}}w(t)=+\infty$ we get
				\begin{equation*}
				\int\limits_{w(t)}^{+\infty}\dfrac{dy}{y^{r}}\leq 2(T^*_{h}-t),
				\end{equation*}
				hence
				\begin{equation*}
				w(t)\geq \left(\dfrac{1}{p-1}\right)^{\frac{2}{p-1}}(T^*_{h}-t)^{\frac{-2}{p-1}}\\
				\Rightarrow \sum\limits_{j=1}^{N}u_{j}^{2}(t)\geq \left(\dfrac{1}{p-1}\right)^{\frac{2}{p-1}}(T^*_{h}-t)^{\frac{-2}{p-1}}.
				\end{equation*}
				But
				\begin{equation*}
				\left(\sum\limits_{j=1}^{N}u_{j}(t)\right)^{2}\geq\sum\limits_{j=1}^{N}u_{j}^{2}(t)\geq \left(\dfrac{1}{p-1}\right)^{\frac{2}{p-1}}(T^*_{h}-t)^{\frac{-2}{p-1}},
				\end{equation*}
				then
				\begin{equation}
				\sum\limits_{j=1}^{N}u_{j}(t)\geq \left(\dfrac{1}{p-1}\right)^{\frac{1}{p-1}}(T^*_{h}-t)^{\frac{-1}{p-1}}.
				\label{estimation3}
				\end{equation}
				In an other hand we have
				\begin{equation*}
				\max\limits_{1\leq j\leq N} u_{j}(t)\geq u_{j}(t)\Rightarrow \max\limits_{1\leq j\leq N} u_{j}(t)\geq \dfrac{1}{N} \sum\limits_{j=1}^{N}u_{j}(t). 
				\end{equation*}
				Therefore \eqref{estimation3} implies
				\begin{equation*}
				\max\limits_{1\leq j\leq N} u_{j}(t)\geq \dfrac{1}{N}\left(\dfrac{1}{p-1}\right)^{\frac{1}{p-1}}(T^*_{h}-t)^{\frac{-1}{p-1}}.  
				\end{equation*}	
				To prove the other inequality, recall the relation \eqref{F}
				\begin{equation*}
				w'(t)\geq 2\tilde{k} w^{\frac{p+1}{2}}(t) 
				\end{equation*}	
				where $\tilde{k}:=\left( \dfrac{p-1}{p+1}-c \left\|U_{h}(0) \right\|^{-\beta}_{p+1}\right)>0$ for a large initial data if $q< \dfrac{2p}{p+1}$ \\and $b_{\infty}<\dfrac{p-1}{2}\bigg(\dfrac{2}{p+1}\bigg)^{1/p+1}$ if $q=\dfrac{2p}{p+1}$.\\
				Integrating again over $[t,T^*_{h})$ and we do the same calculations as before, we get
				\begin{eqnarray*}
					&& w(t)\leq \left(\dfrac{1}{k(p-1)}\right)^{\frac{2}{p-1}} \left(T^*_{h}-t\right)^{\frac{-2}{p-1}}.\\
					&\Rightarrow& \sum\limits_{j=1}^{N} u_{j}^{2}(t)\leq \left(\dfrac{1}{k(p-1)}\right)^{\frac{2}{p-1}} \left(T^*_{h}-t\right)^{\frac{-2}{p-1}}.
				\end{eqnarray*} 
				we use now that
				\begin{equation*}
				\max\limits_{1\leq j\leq N}u_{j}(t)\leq\left(\sum\limits_{j=1}^{N} u_{j}^{2}(t)\right)^{\frac{1}{2}}\leq \left(\dfrac{1}{k(p-1)}\right)^{\frac{1}{p-1}} \left(T^*_{h}-t\right)^{\frac{-1}{p-1}}
				\end{equation*}	
				Therefore we obtain
				\begin{equation*}
				\max\limits_{1\leq j\leq N} u_{j}(t)\leq\left(\dfrac{1}{k(p-1)}\right)^{\frac{1}{p-1}} \left(T^*_{h}-t\right)^{\frac{-1}{p-1}}
				\end{equation*}
				This finishes the proof of the estimation of the numerical blow up rate.
			\end{proof}
			In the next section, we prove that $U_{h}$ blows up in $l^{2}$ norm and we give an estimation of the blow up time.
		\section{Estimation of the numerical blow up time}
		\noindent Let $k:= 2\left( \dfrac{p-1}{p+1}-c \left\|U_{h}(0) \right\|^{-\beta}_{p+1}\right)>0$ and we consider the solution of the equation
			\begin{equation}
			R'(t)= k R^{\frac{p+1}{2}}(t), \ \ \ \ \ \ R(0)=R_{0}\leq F(0)=\left\| U_{h}(0)\right\|^{2}_{2}.
			\label{R}
			\end{equation}
			Clearly, the solution $R$ blows up in a finite time $T$.	
		
			Using \eqref{R} we get 
			\begin{eqnarray*}
		\nonumber	T&=& \int\limits_{0}^{T}dt\\
		\nonumber	&=& \int\limits_{0}^{T}\dfrac{R'(t)}{k R^{\frac{p+1}{2}}(t)}dt\\
		\nonumber	&=& \int\limits_{R_{0}}^{+\infty}\dfrac{d\xi}{k \xi^{\frac{p+1}{2}}}\\
			&=& \dfrac{1}{(p-1)\left( \dfrac{p-1}{p+1}-c \left\|U_{h}(0) \right\|^{-\beta}_{p+1}\right)R_{0}^{\frac{p-1}{2}}}
			\label{23}
			\end{eqnarray*}
			with $\lim\limits_{t\longrightarrow T}R(t)=+\infty.$\\
			Finally, using the relation \eqref{F} and the theory of differential inegality, we know that if $F(0)\geq R(0)$ then $F(t)\geq R(t)$ for all $t\in(0,T)$. Hence if we choose $R_{0}=F(0)=\left\|U_{h}(0)\right\|_{2}^2$, then \eqref{F} implies that $F(t)\longrightarrow +\infty$ as ${t\longrightarrow T^*_{h}}$ and $T^*_{h}$ is estimated by:
			\begin{equation*}
			T^*_{h}\leq \dfrac{1}{(p-1)\left( \dfrac{p-1}{p+1}-c \left\|U_{h}(0) \right\|^{-\beta}_{p+1}\right)\left\|U_{h}(0)\right\|_{2}^{p-1}}. 
			\end{equation*}
			Note that the relation $\left\|.\right\|_{2}\leq \left\|.\right\|_{\infty}$ implies that if the solution blows up in $l^{2}$ norm, then it blows up in $l^{\infty}$ norm.\\
		In an other hand, using relation \eqref{estimate} and integration over $[0,T^*_{h})$ we get
		 \begin{equation*}
		 T^*_{h}\geq \dfrac{1}{(p-1)\left(\sum\limits_{j=1}^{N}u_{j}^{2}(0)\right)^{\frac{p-1}{2}}}.
		 \end{equation*}
\section{Numerical experiments}
In this section, we give some computational results concerning the blow up of the numerical solution and the nonincreasing on the numerical energy $J$ .\\
We study also, the effect of the function $b$ and the parameter $q$ on the behavior of the solution.\\
In a first step, we will take $b$ as a positive constant and we will study its effect on the behavior of the numerical solution. More precisely, we will confirm the theoretical results proved in \cite{alfonsi} concerning the upper bound of $b_\infty$ when $q=\dfrac{2p}{p+1}$. \\
In a second step, we will take $b$ as a positive and continuous function and we will study its effect on the behavior of the numerical solution. Concerning the theoretical results of this case, it will be done in a next paper.\\
\newline
In articles \cite{khe} and \cite{hani}, we took a positive and symmetric initial data and we have proved that solution blows up. In order to prove that symmetry has no effect on the blowing up result, we take in figure \ref{fig1}, $u_{0}(x)=10^3x^2(1-x^2)\exp(x-1)$ which is a nonnegative and nonsymmetric function satisfying $J(u_0)<0$. We will prove that the solution blows up in a finite time which confirms the result of this paper and that blowing up occurs even if the initial data is symmetric or not. We prove also that blowing up is localised in the maximum point, which confirms the result of Theorem 1.4 in \cite{hani}.
\begin{figure}[H]
\centering	
\includegraphics[width=\textwidth,height=6cm]{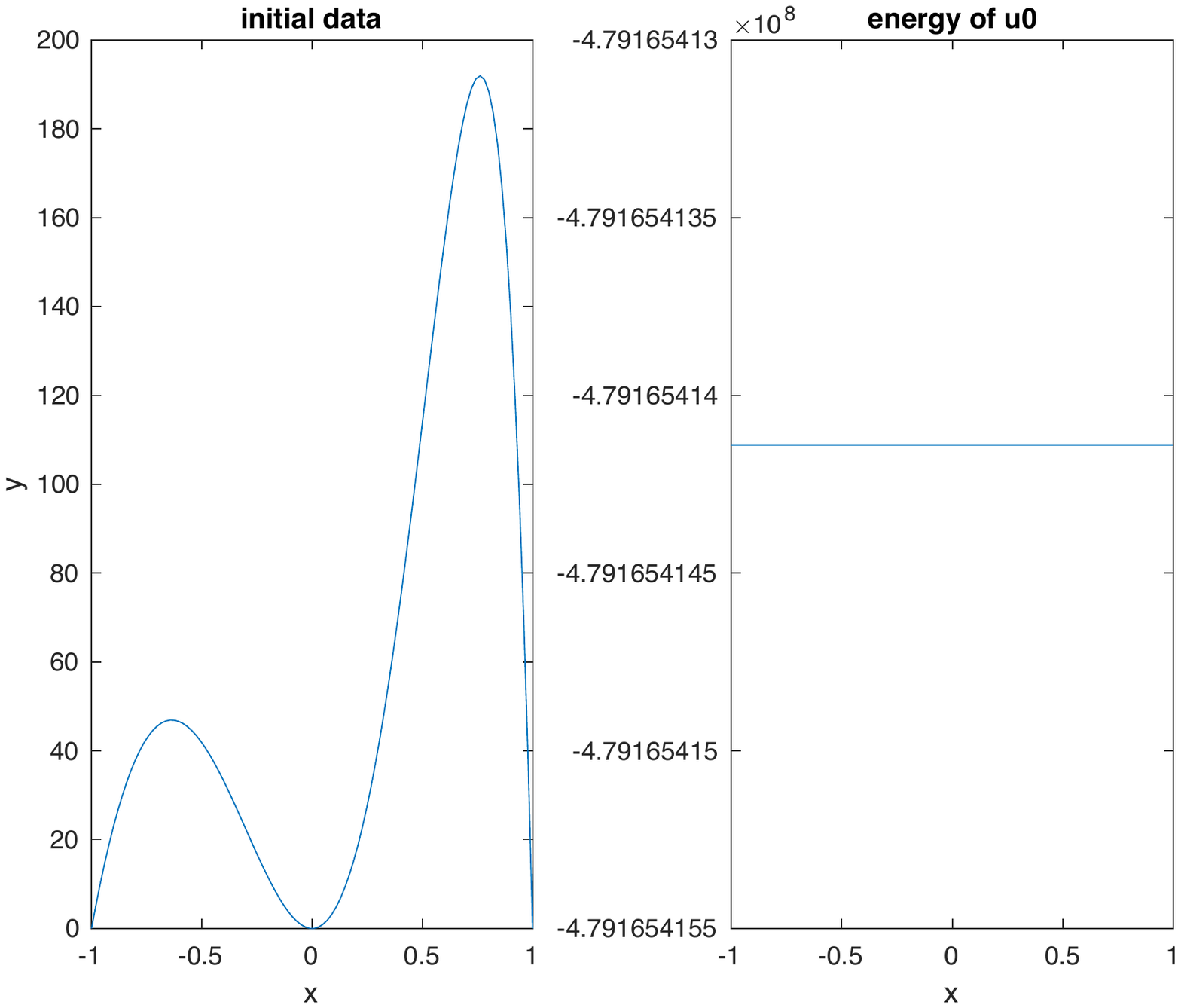}
\caption{Initial data: positive, nonsymmetric with a negative energy}
\label{fig1}
\end{figure} 
Next, we choose, $p=3$, $q=1.3<\dfrac{2p}{p+1}$ and $b=1$.  
In figure \ref{fig2},  we can see that the solution blows up on the maximum point and in figure \ref{fig3}, we prove the nonincreasing of the energy .
\begin{figure}[H]
			\centering
\includegraphics[width=\textwidth,height=6cm]{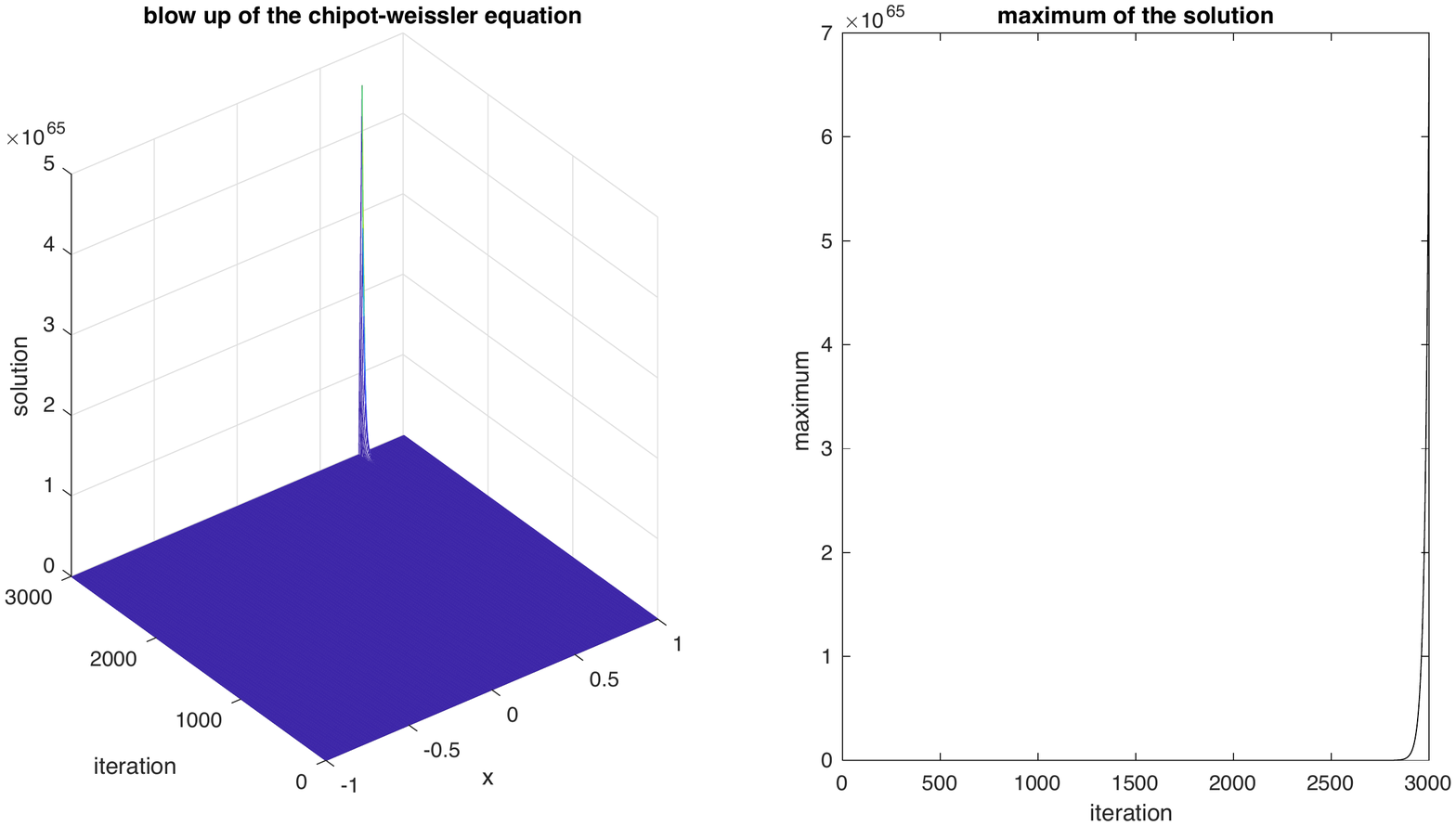}
			\caption{Blow-up of the numerical solution for a nonsymmetric initial data}
\label{fig2}
\end{figure} 		
\begin{figure}[H]
\centering
\includegraphics[width=12cm,height=6cm]{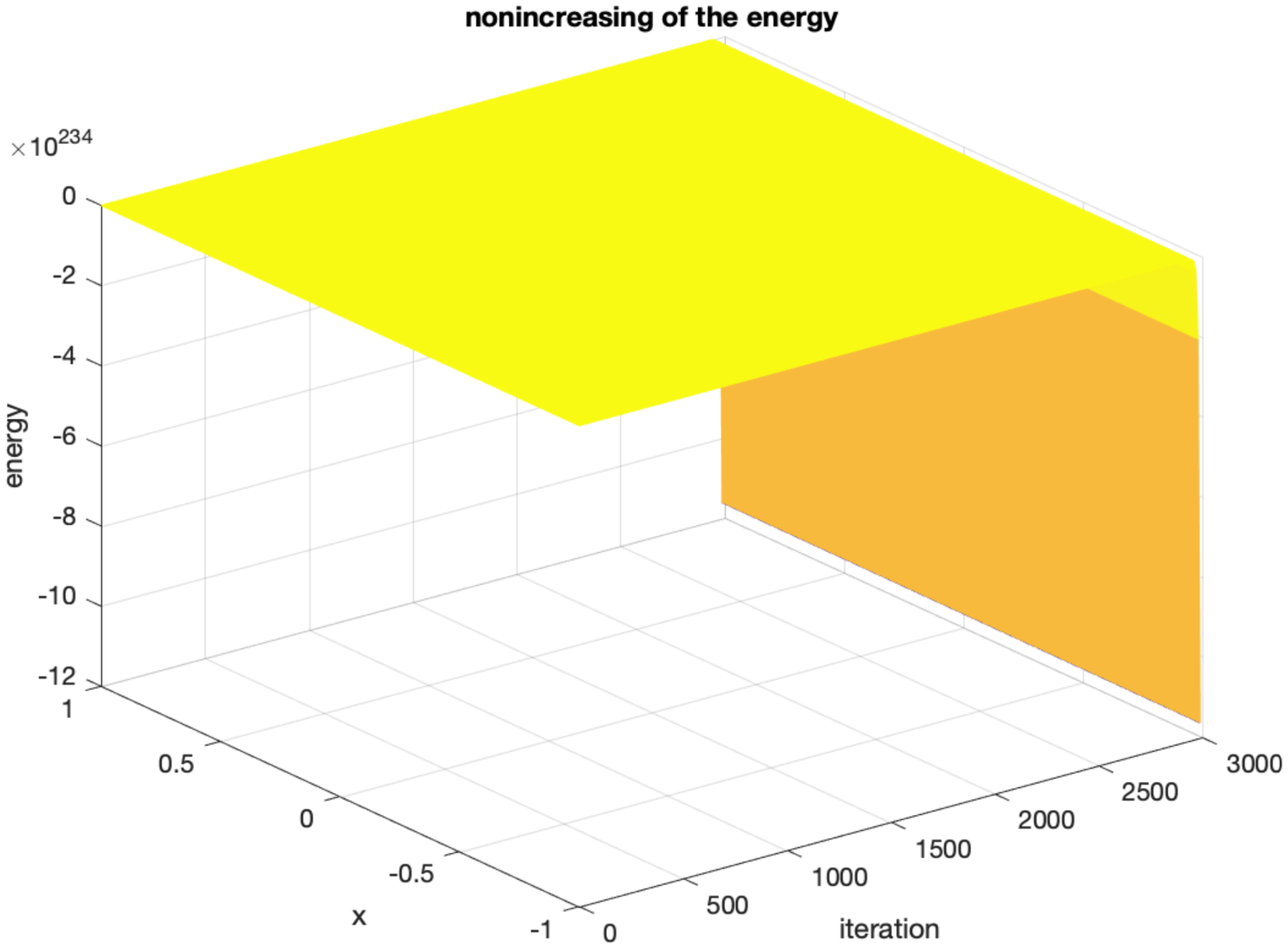}
\caption{Nonincreasing of the energy}
\label{fig3}
\end{figure} 
Note that figures \ref{fig1}, \ref{fig2} and \ref{fig3} complete results of \cite{khe} and \cite{hani}, so we prove that the solution blows up in the maximum point whether the solution is symmetric or not.\\
Next we choose a symmetric initial data $u_{0}(x)=10^3\sin(\dfrac{\pi}{2}(x+1))$ (see figure  \ref{fig4} )we will study the numerical effect of $b$ when $q=\dfrac{2p}{p+1}$ and $q<\dfrac{2p}{p+1}$.
\begin{figure}[H]
	\centering	
	\includegraphics[width=\textwidth,height=6cm]{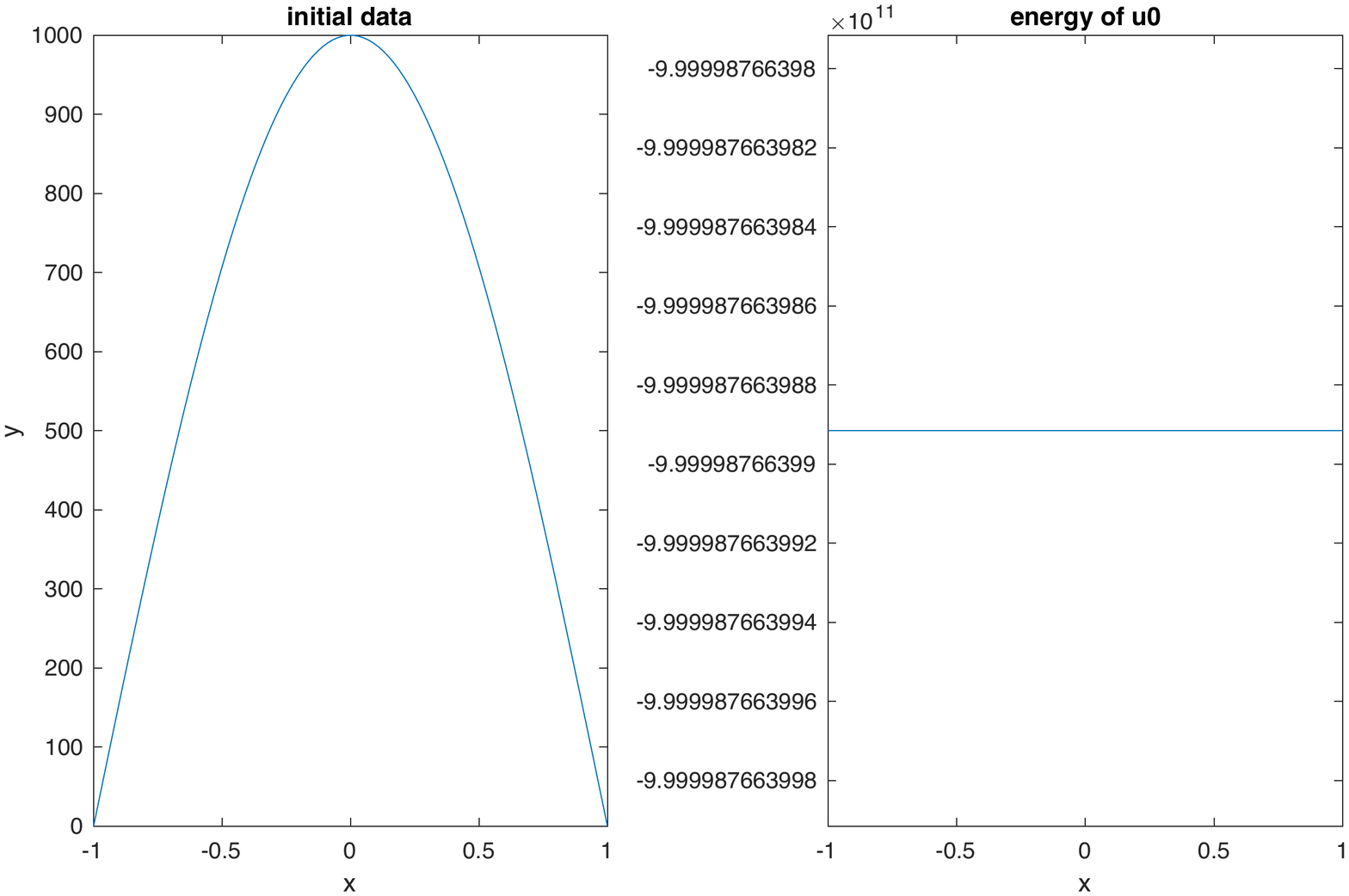}
	\caption{Positive and symmetric initial data with a negative energy }
	\label{fig4}
\end{figure} 
In figure \ref{fig5}, for $b=1$ and $q<\dfrac{2p}{p+1}$  we can see that the numerical solution blows up and in figure \ref{fig6}, we prove the nonincreasing of the energy.	
\begin{figure}[H]
	\centering
	\includegraphics[width=\textwidth,height=6cm]{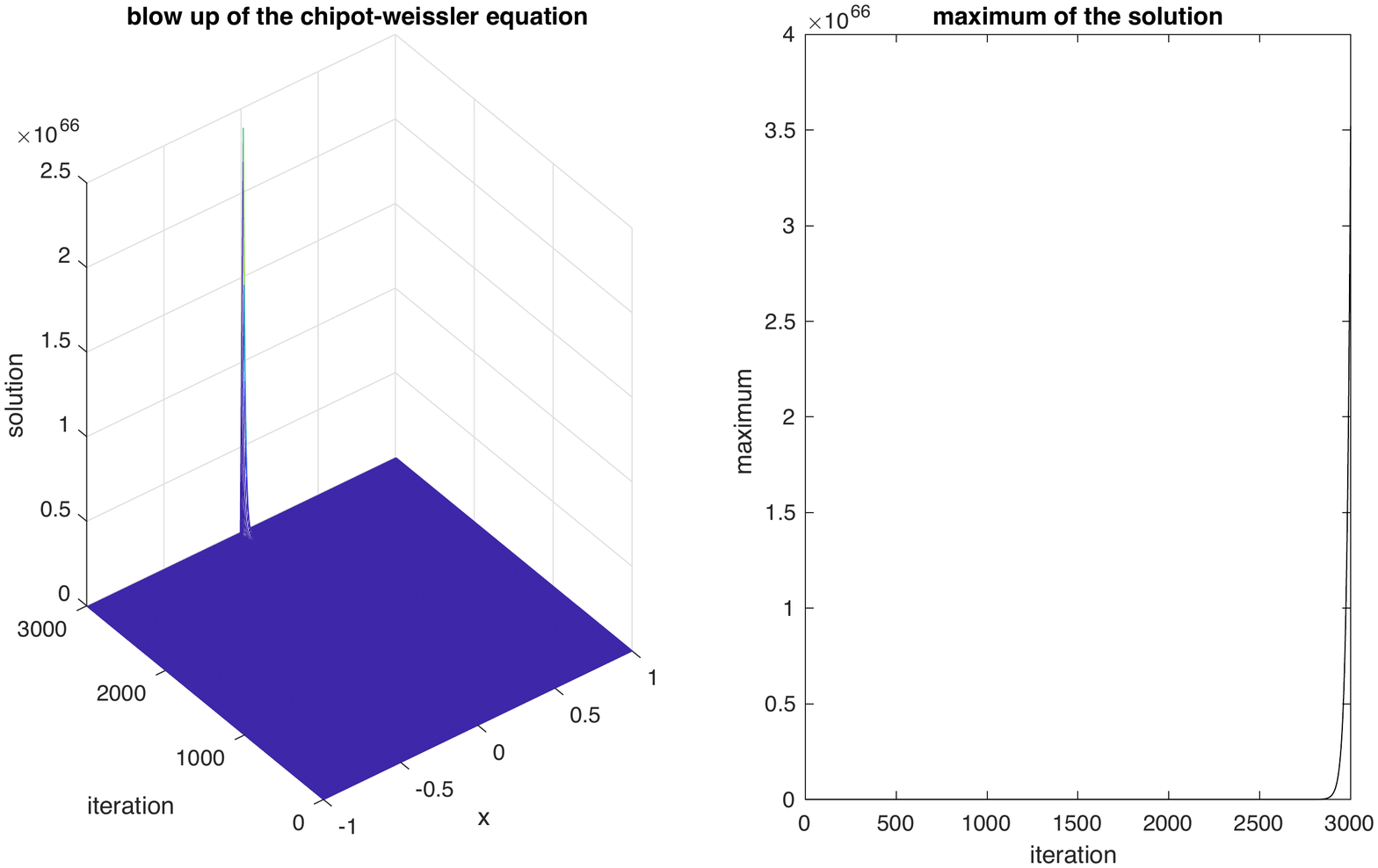}
	\caption{Blow-up of the numerical solution for a symmetric initial data for $b=1$ and $q<\dfrac{2p}{p+1}$  }
	\label{fig5}
\end{figure} 	
\begin{figure}[H]
	\centering
	\includegraphics[width=12cm,height=6cm]{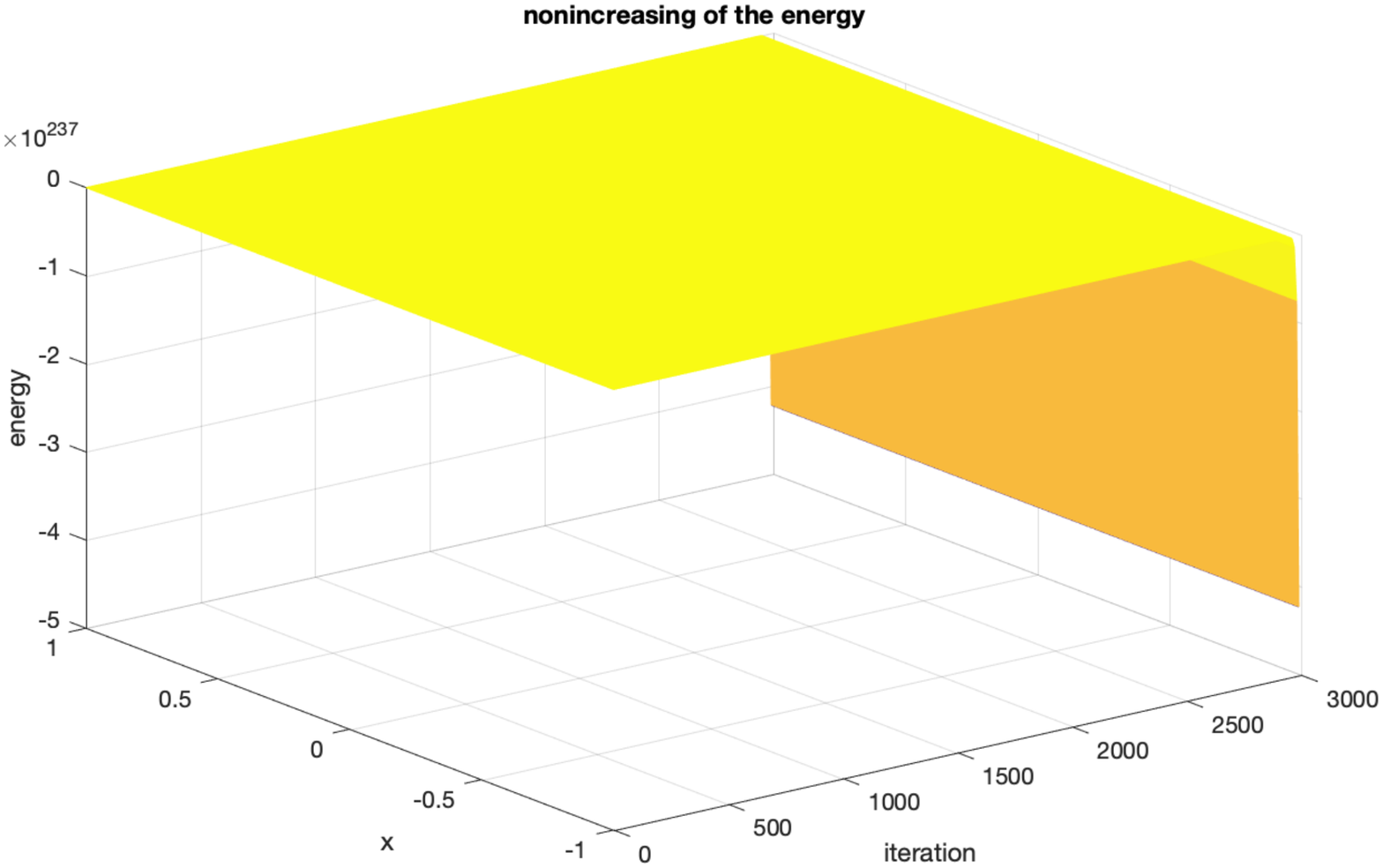}
	\caption{Nonincreasing of the energy}
	\label{fig6}
\end{figure} 
Next, if we take $q=\dfrac{2p}{p+1}$ we get figure \ref{fig7}
\begin{figure}[H]
	\centering
	\includegraphics[width=12cm,height=6cm]{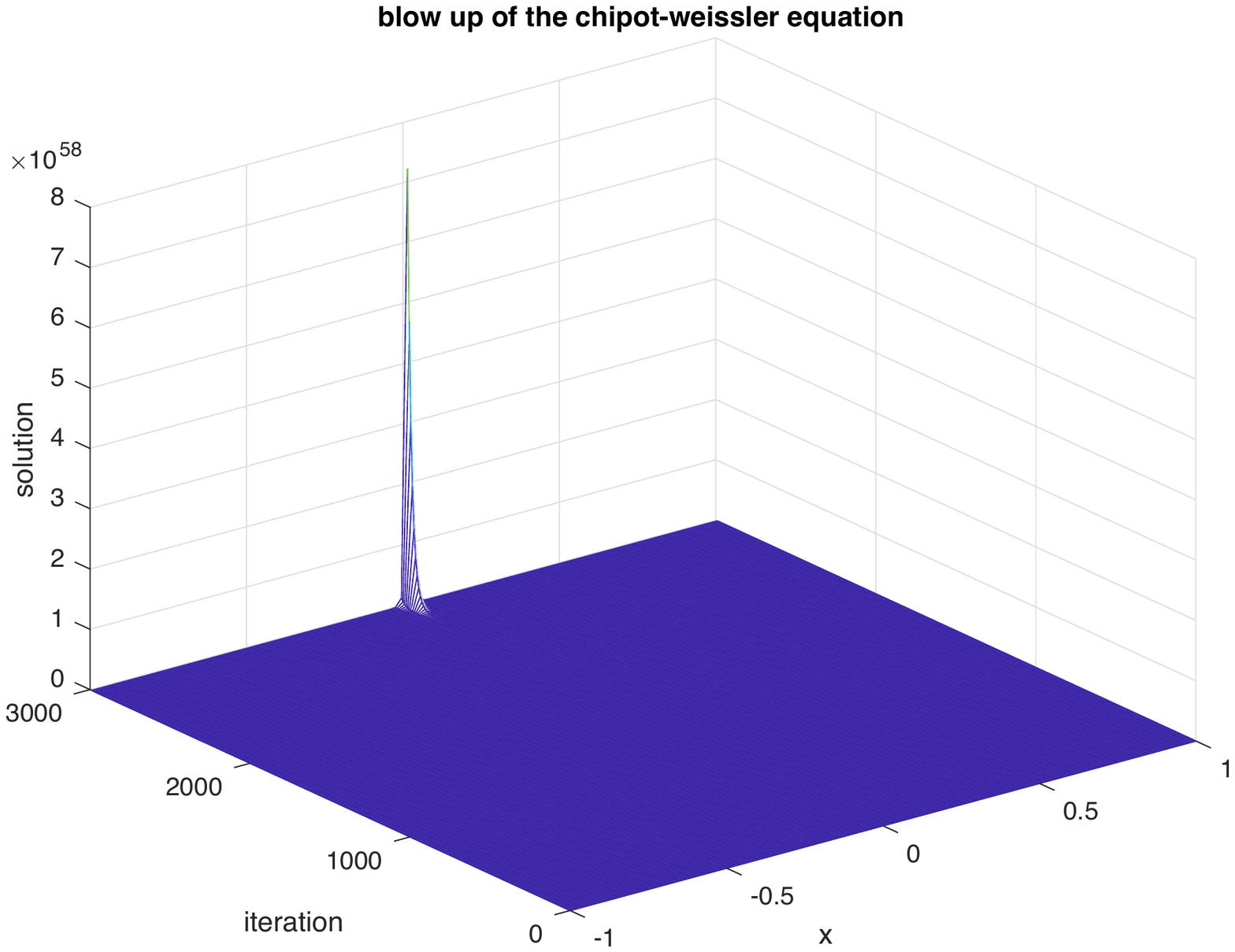}
	\caption{blow up of the numerical solution $b=1$ and $q=\dfrac{2p}{p+1}$ }
	\label{fig7}
\end{figure} 
If we compare  figure \ref{fig5} and figure  \ref{fig7}, we can see the damping effect of the gradient term for $b=1$: the maximum of the numerical solution when $q<\dfrac{2p}{p+1}$ is greater than the maximum when $q=\dfrac{2p}{p+1}$.\\
In order to study the effect of the constant $b$, we will prove numerically that $b$ has no effect when $q<\dfrac{2p}{p+1}$ contrary to the case $q=\dfrac{2p}{p+1}$.\\
Let $p=3$ and $q=1.3<\dfrac{2p}{p+1}$, in figures \ref{fig8}, \ref{fig9}, \ref{fig10} and \ref{fig11} we take $b=1$, $b=10$, $b=100$ and $b=1000$, we can see that the numerical solution has the same profile for different values of $b$, which confirms that $b$ has no effect when $q< \dfrac{2p}{p+1}$.\\
\begin{minipage}{0.56\textwidth}
	\begin{figure}[H]
		\includegraphics[width=\textwidth,height=6cm]{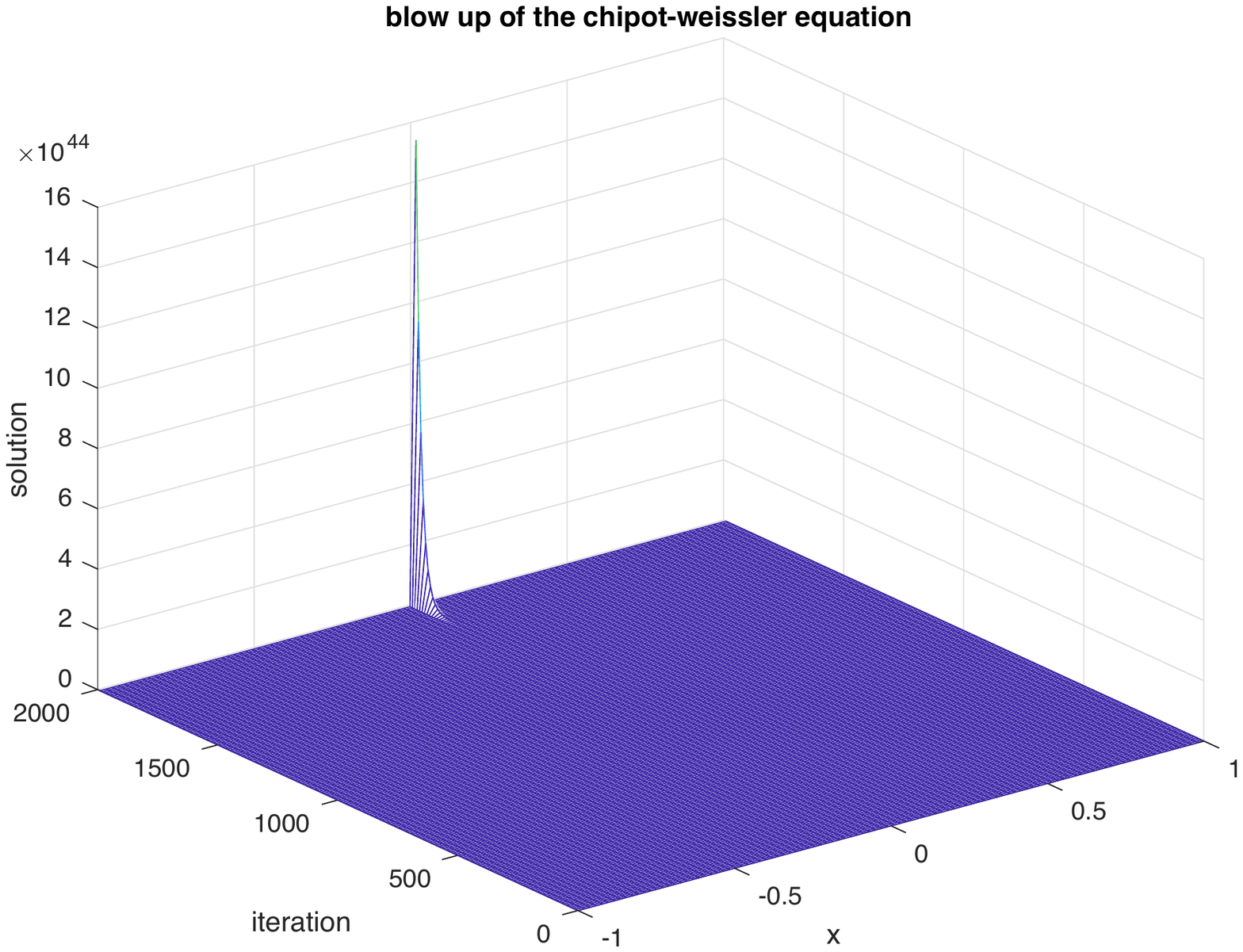}
	    \caption{blow up of the solution for $b=1$}
		\label{fig8}
	\end{figure}
\end{minipage}
\hspace{5ex} 
\begin{minipage}{0.56\textwidth}
	\begin{figure}[H]
		\includegraphics[width=\textwidth,height=6cm]{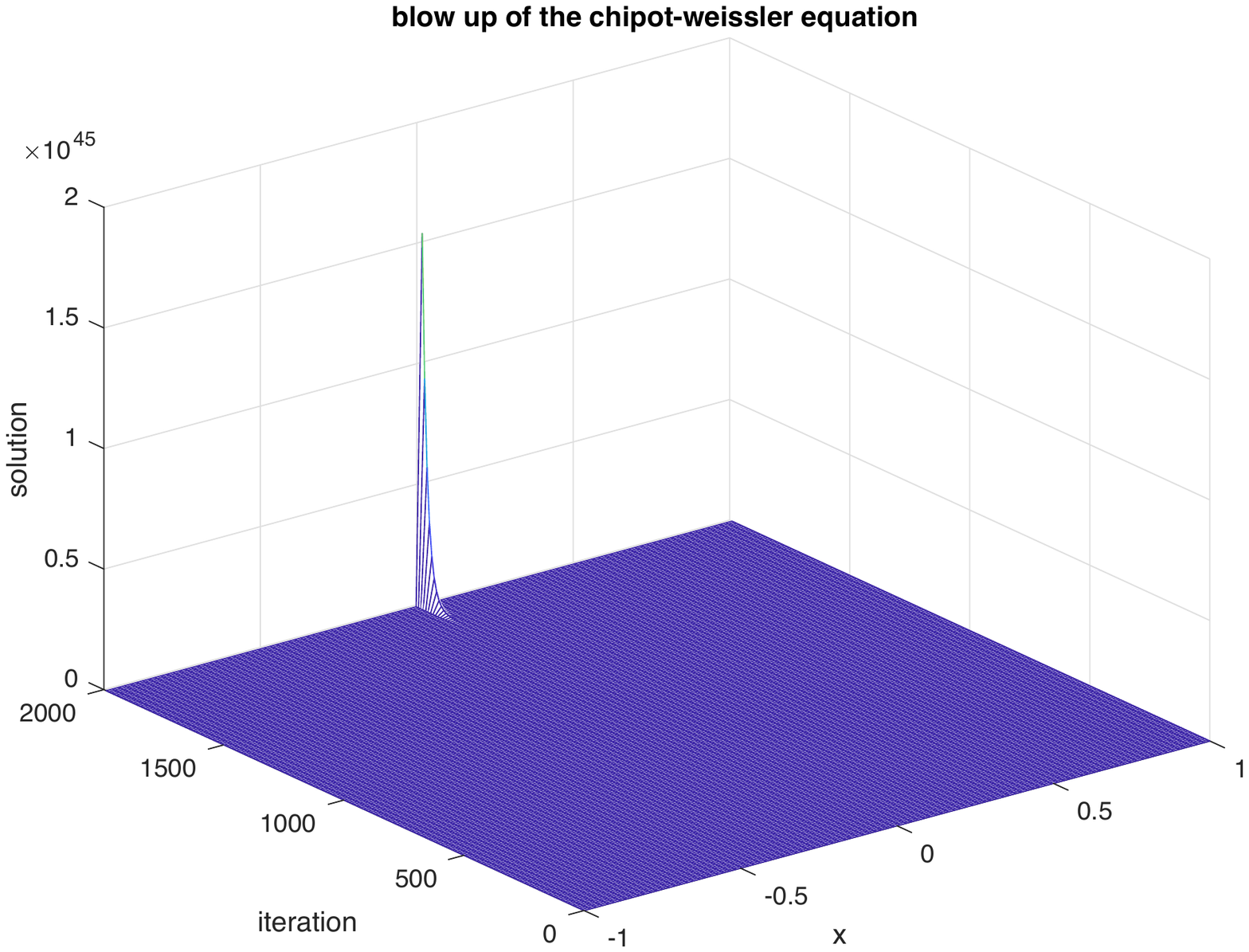}
	\caption{blow up of the solution for $b=10$}
		\label{fig9}
	\end{figure}
\end{minipage}\\
\begin{minipage}{0.56\textwidth}
	\begin{figure}[H]
		\includegraphics[width=\textwidth,height=6cm]{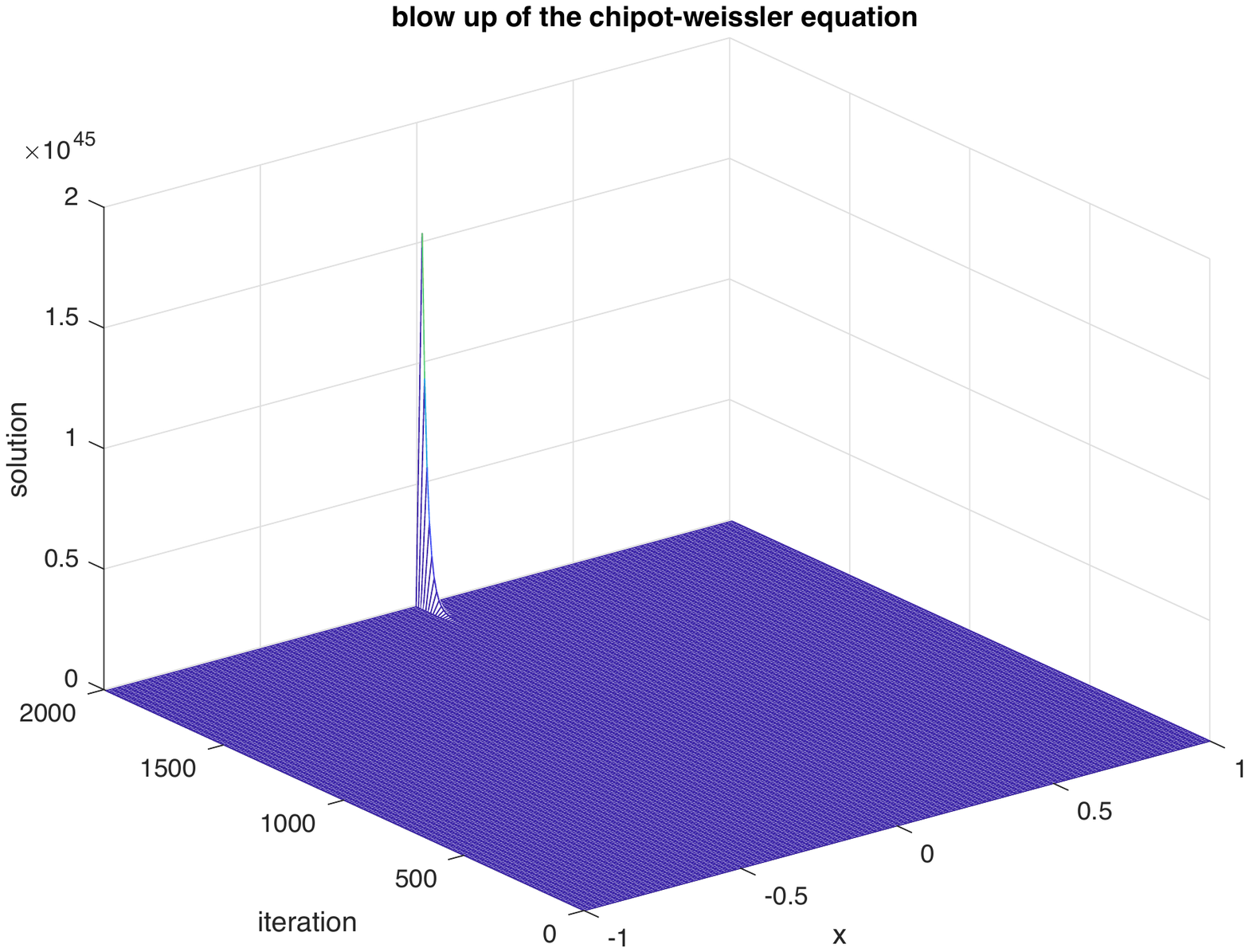}
		\caption{blow up of the solution for $b=100$}
		\label{fig10}
	\end{figure}
\end{minipage}
\hspace{5ex} 
\begin{minipage}{0.56\textwidth}
	\begin{figure}[H]
		\includegraphics[width=\textwidth,height=6cm]{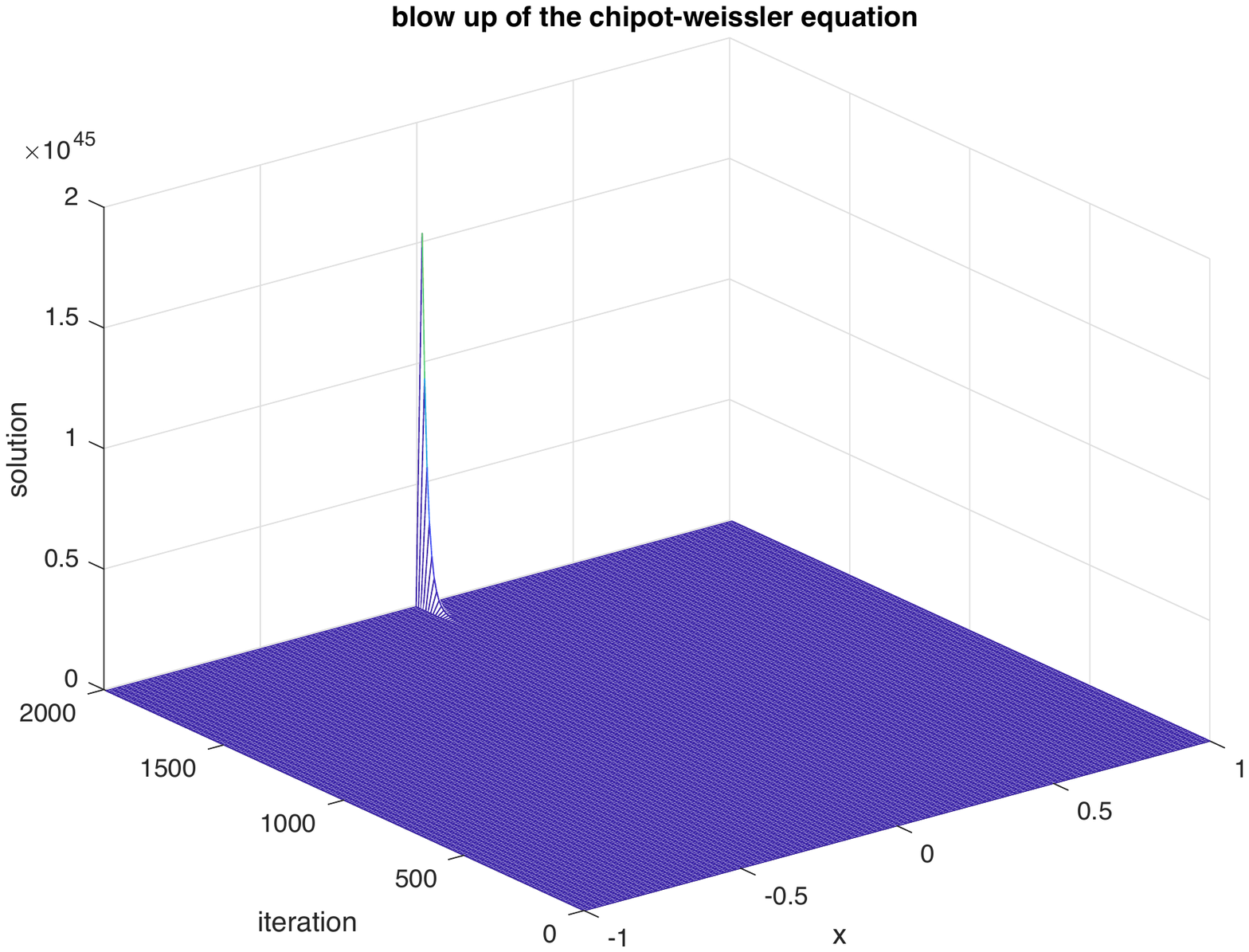}
		\caption{blow up of the solution for $b=1000$}
		\label{fig11}
	\end{figure}
\end{minipage}
Let now study the effect of $b$ for $q=\dfrac{2p}{p+1}$. In \cite{alfonsi}, authors proved the next theorem:
\begin{th1}
Let $\Omega \subset\mathbb{R}$, $1<p<5$, $q=\dfrac{2p}{p+1}$ and $0<b<\dfrac{p-1}{2}\bigg(\dfrac{2}{p+1}\bigg)^{1/p+1}$, then solution blows up in finite time for a positive initial data $u_0$ sufficiently regular satisfying $E(u_0)<0$ and $\Delta u_0-b|\nabla u_0|^q+u^p_0\geq 0.$
\end{th1}
In this paper we choose $p=3$. In figures \ref{fig12}, \ref{fig13} and \ref{fig14}, we take $b=1$, $b=1.48$ and $b=1.49$. Remark that $b=1.48$ is the greatest value of $b$ where the solution conserve positivity. For $b>1.48$, solution becomes nonpositive: there exists $x^*$ such that $u(x^*)<0$. Which confirms the result of the above theorem.\\
\begin{minipage}{0.56\textwidth}
	\begin{figure}[H]
	\includegraphics[width=\textwidth,height=6cm]{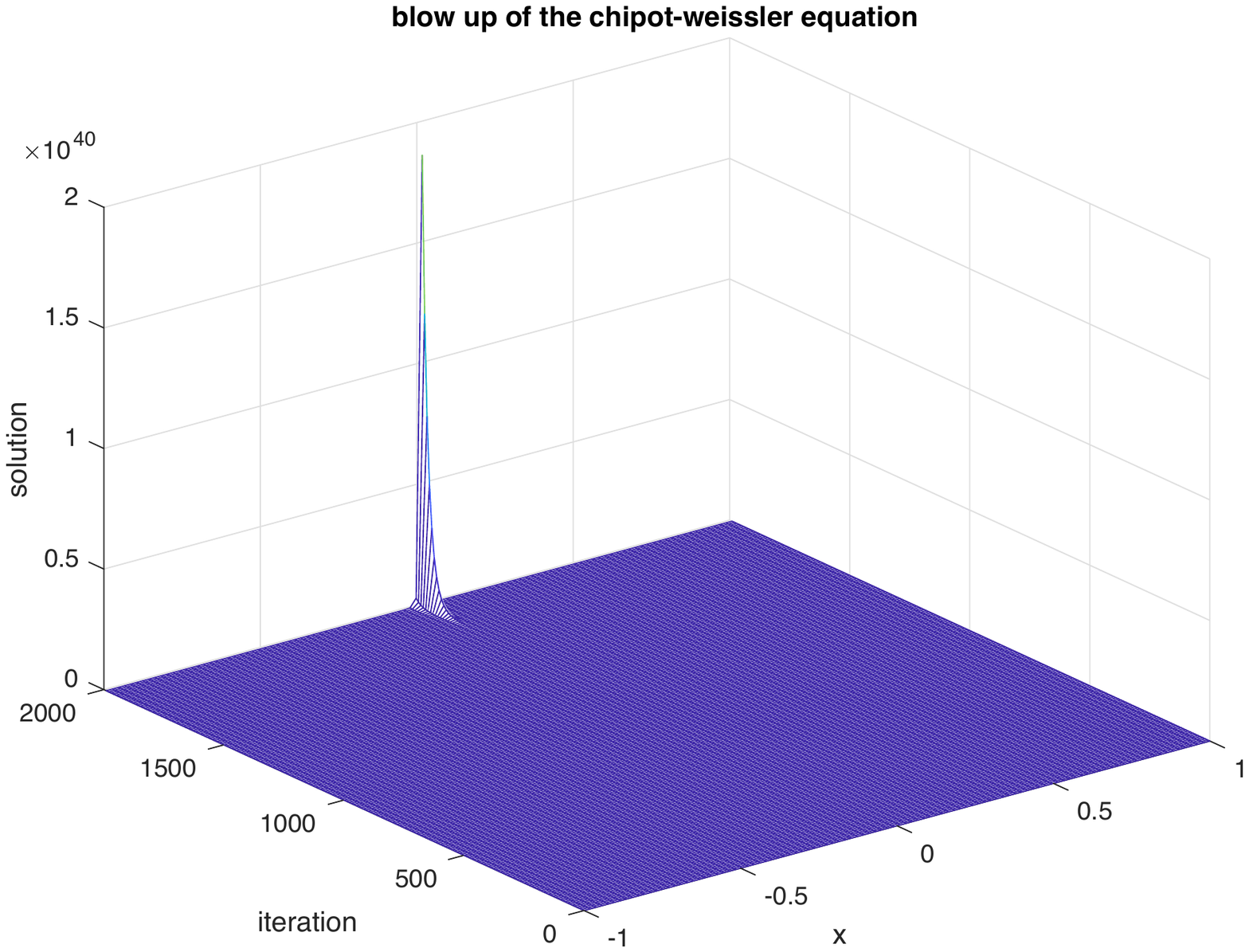}
	\caption{behavior of the numerical solution for $b=1$}
	\label{fig12}
\end{figure}
\end{minipage}
\begin{minipage}{0.56\textwidth}
	\begin{figure}[H]
	\includegraphics[width=\textwidth,height=6cm]{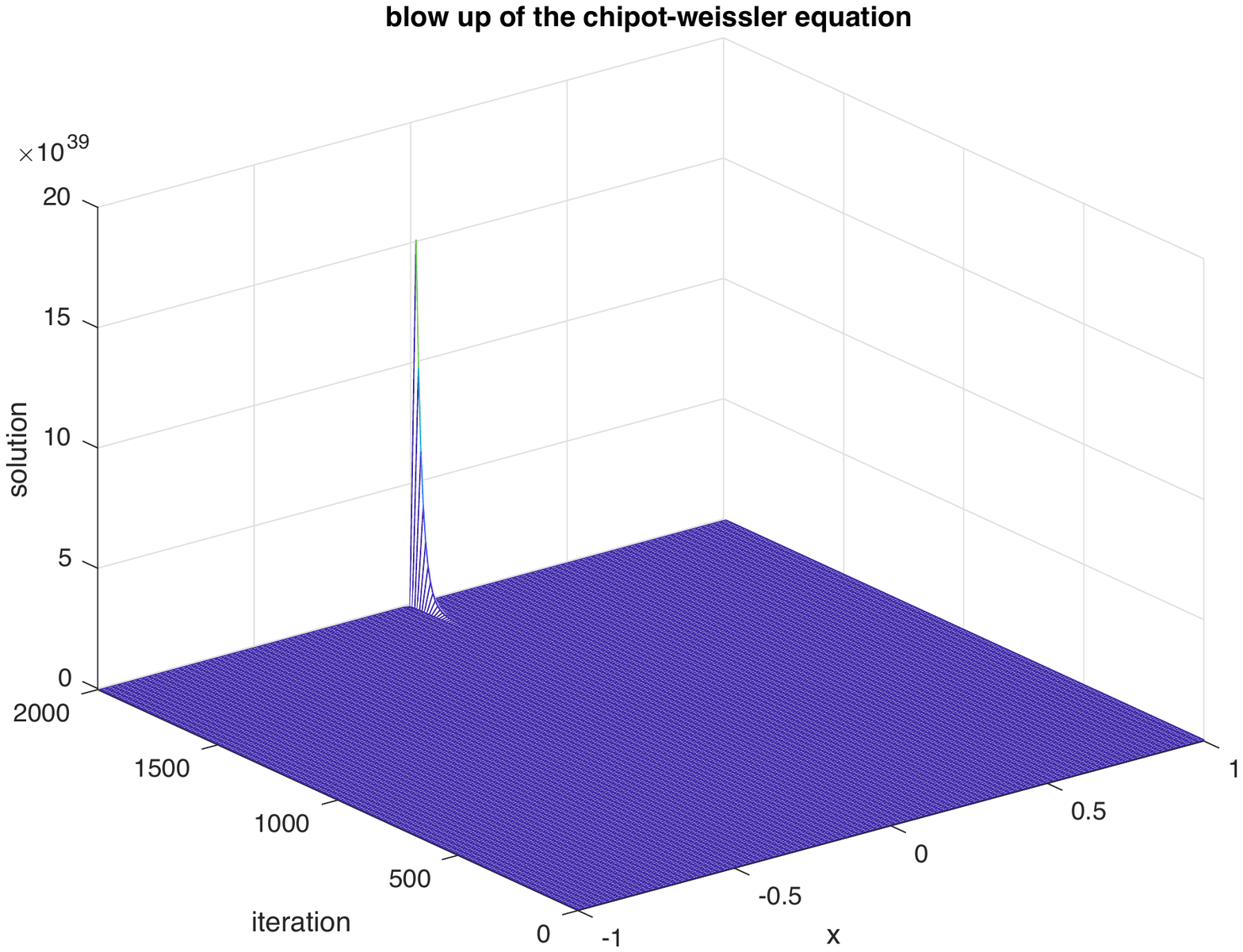}
	\caption{behavior of the numerical solution for $b=1.48$}
	\label{fig13}
\end{figure}
\end{minipage}
\begin{center}
\begin{minipage}{0.56\textwidth}
	\begin{figure}[H]
		\centering
	\includegraphics[width=12cm,height=6cm]{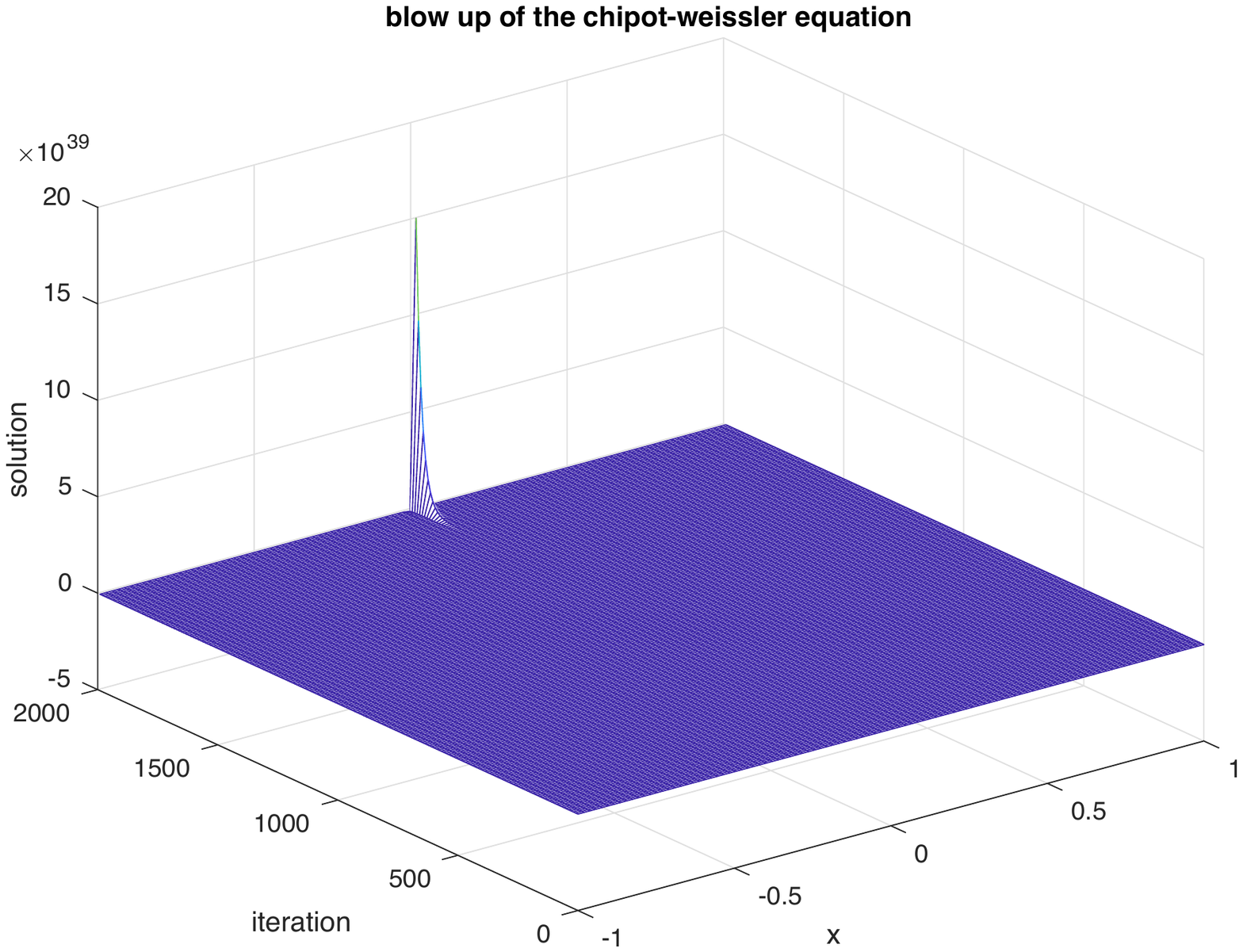}
	\caption{behavior of the solution for $b=1.49$}
	\label{fig14}
\end{figure}
\end{minipage}
\end{center}
Finally, let $b$ be a continuous, positive and bounded function. We will study the effect of the function $b$ on the behavior of the numerical solution when $b_{\infty}$ is small and large.\\
In figures \ref{fig15} and \ref{fig16}, and for $p=3$, $q=\dfrac{2p}{p+1}$, we take $b(x)=\exp(-x^3)$ then $b(x)=10^3\exp(x^3)$ respectively. We study the behavior of the numerical solution for a different iterations in time.  We can see that for $b_{\infty}$ sufficiently large, solution becomes negative, but for $b_{\infty}$ sufficiently small, the numerical solution has the same properties as the exact one, which proves that for $q=\dfrac{2p}{p+1}$, the function $b$ has an effect on the behavior of the numerical solution.\\
\begin{minipage}{0.56\textwidth}
	\begin{figure}[H]
		\includegraphics[width=\textwidth,height=6cm]{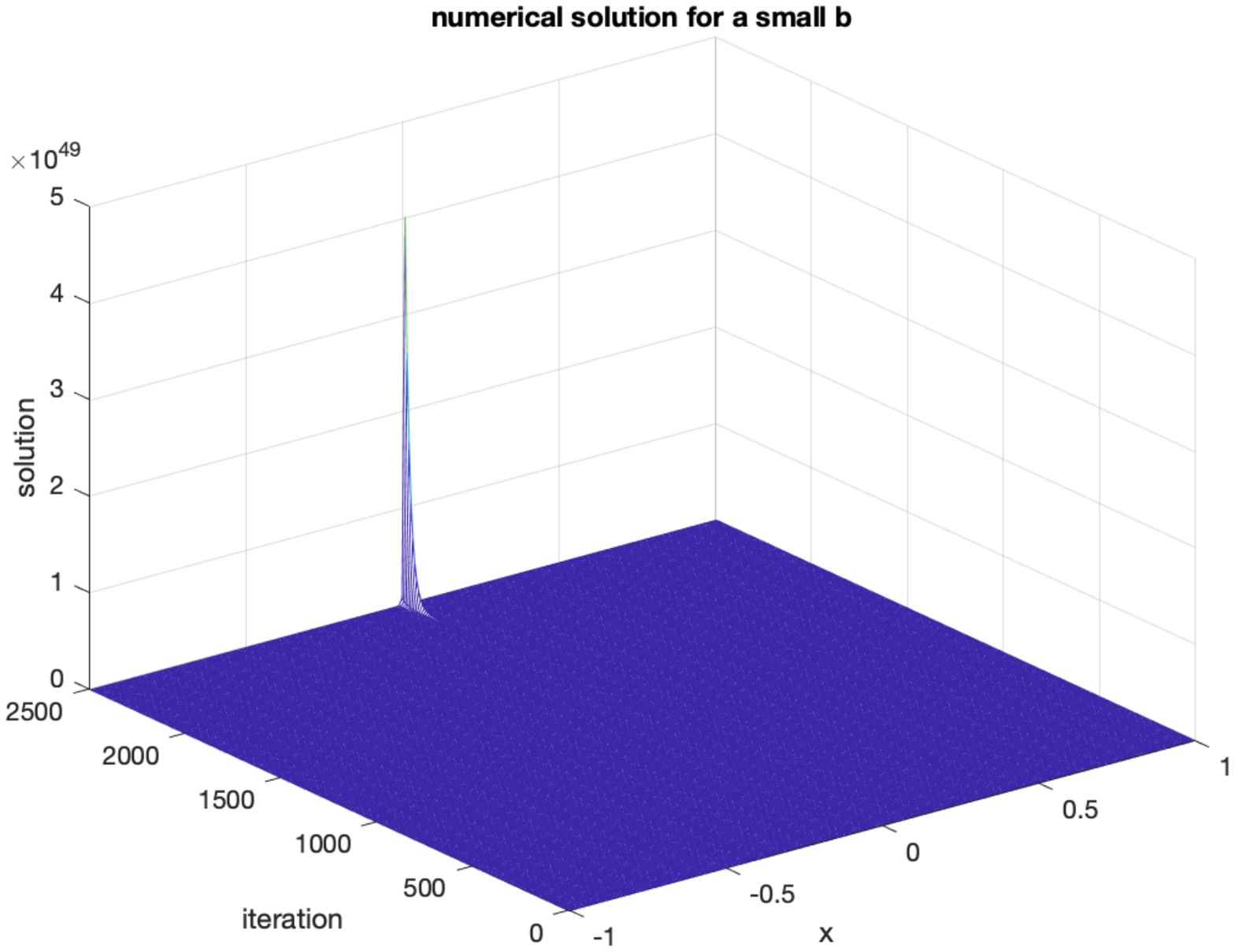}
		\caption{behavior of the numerical solution for $b_{\infty}$ sufficiently small}
		\label{fig15}
	\end{figure}
\end{minipage}
\begin{minipage}{0.56\textwidth}
	\begin{figure}[H]
		\includegraphics[width=\textwidth,height=6cm]{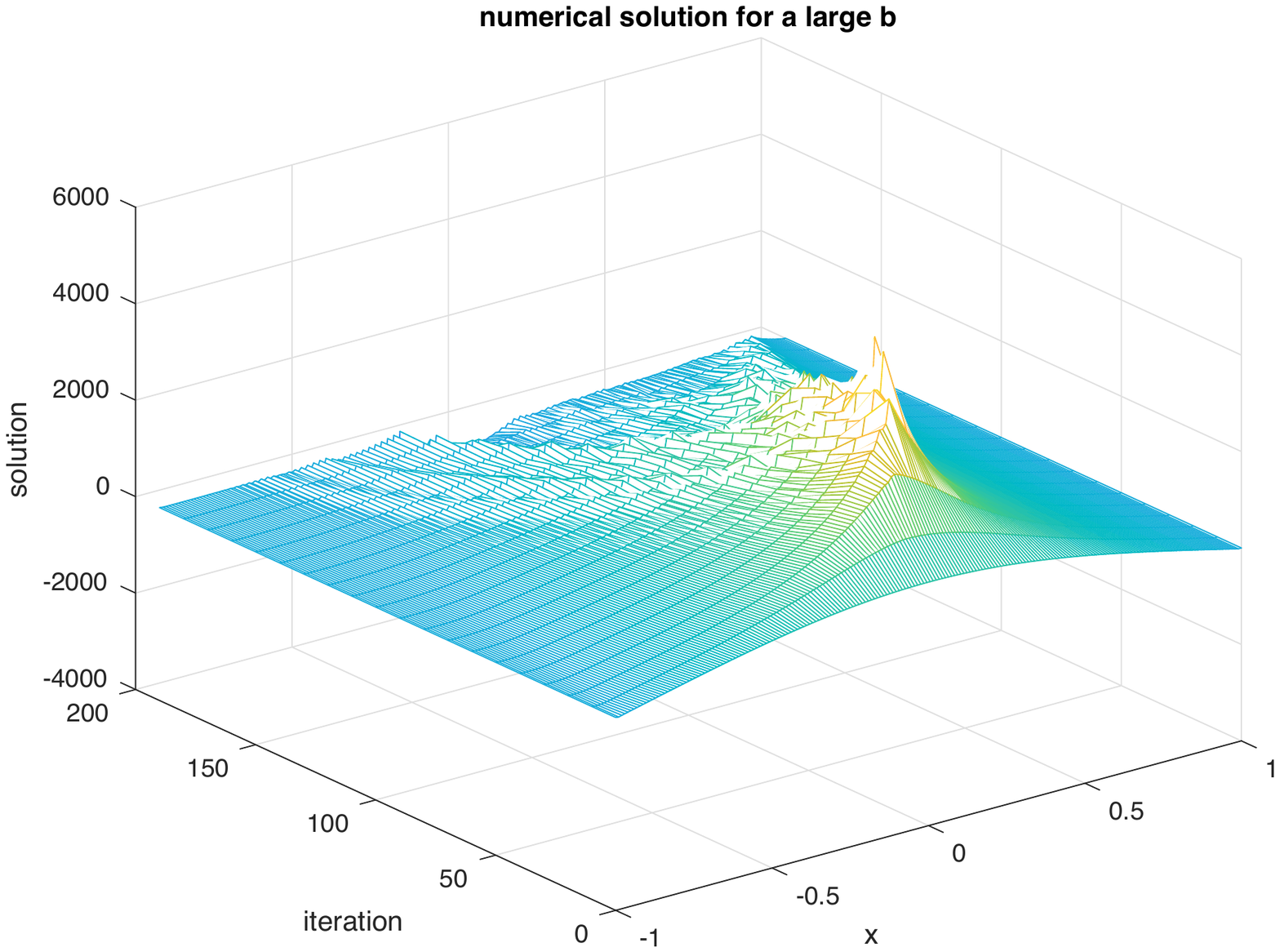}
		\caption{behavior of the numerical solution for $b_{\infty}$ large}
		\label{fig16}
	\end{figure}
\end{minipage}
\vspace{1cm}
\begin{center}
	\textbf{Conclusion and open problems}
\end{center}
In this paper, we have considered a nonlinear parabolic problem. We have proved that the numerical solution blows up in a finite time.\\
Problem \eqref{exacte} was studied by many authors in the case $b\in \mathbb{R}$. For example, Tayachi, Zaag, Weissler and others have studied the self-similar blow up profiles (see \cite{soupletweissler} and the references therein). Also, the blow up set was investigated in \cite{kawohl}. And the blow up rate was studied in \cite{souplettayachi}.\\ 
All these results was proved theoretically. In our last papers \cite{khe} and \cite{hani}, we have studied numerically the blow up rate and the blow up set.\\
In a next paper, we are studying the numerical self-similar blow up profile and we are trying to give answers to these questions:
\begin{itemize}
		\item The existence of blow up solutions for problem \eqref{exacte} in the case $q=\dfrac{2p}{p+1}$ is known (theoretically and numerically) only for $0\leq b<\dfrac{p-1}{2}\bigg(\dfrac{2}{p+1}\bigg)^{1/p+1} $. What happens when $b>\dfrac{p-1}{2}\bigg(\dfrac{2}{p+1}\bigg)^{1/p+1}$?
	\item Does the numerical profile is the same as the theoretical one?
	\item In \cite{soupletweissler}, authors proved the existence of non trivial backward self similar solution for $0<b\leq 2$ and $q=\dfrac{2p}{p+1}$. they have no idea about the profile when $b>2$. Can we provide a numerical result in this case?
	\item What happens if we remplace $b(x)$ with $f(u)$?
\end{itemize} 

\end{document}